\documentclass[11pt]{article}

\usepackage[utf8]{inputenc}
\usepackage[width=17cm,height=21cm]{geometry}
\usepackage[many]{tcolorbox}
\usepackage{color}
\usepackage{graphicx}
\usepackage{pdfpages}
\usepackage{hyperref}
\usepackage[capitalise]{cleveref}
\usepackage{amssymb,amsthm,amsmath}
\usepackage{nicefrac}
\usepackage{stmaryrd}
\usepackage{authblk}
\usepackage{float}
\usepackage{subcaption}
\usepackage[ruled,vlined]{algorithm2e}
\usepackage[pagewise]{lineno}

\newcommand{\R}{\mathbb{R}}
\newcommand{\Z}{\mathbb{Z}}
\newcommand{\e}{\varepsilon}
\newcommand{\di}[1]{\,\mathrm{d}#1}
\newcommand{\dive}{\operatorname{div}}
\newcommand{\interior}{\operatorname{int}}
\newcommand{\ddt}{\frac{\operatorname{d}}{\operatorname{d}t}}
\newcommand{\twosc}{\stackrel{2}{\rightharpoonup}}

\newcommand{\thetaff}{\theta^{ff}}
\newcommand{\me}{\textcolor{black}}
\newcommand{\tf}{\textcolor{black}}
\newcommand{\menew}{\textcolor{black}}
\newcommand{\tfnew}{\textcolor{black}}
\newcommand{\tflast}{\textcolor{black}}

\newcommand\restr[2]{{
  \left.\kern-\nulldelimiterspace 
  #1 
  \vphantom{\big|} 
  \right|_{#2} 
  }}

\newtheorem{theorem}{Theorem}

\newtheorem{lemma}{Lemma}

\newtheorem{remark}{Remark}

\crefname{lemma}{Lemma}{lemmas}
\Crefname{lemma}{Lemma}{Lemmas}
\crefname{thm}{theorem}{theorems}
\Crefname{thm}{Theorem}{Theorems}

\tikzset{every picture/.style={line width=0.75pt}}

\newtcolorbox{mybox}[1]{%
    tikznode boxed title,
    enhanced,
    arc=0mm,
    interior style={white},
    attach boxed title to top left= {yshift=-\tcboxedtitleheight/2-0.05cm, xshift=0.7cm},
    fonttitle=\small\bfseries,
    colbacktitle=white,coltitle=black,
    boxed title style={size=small,colframe=white,boxrule=0pt},
    title={#1}}

\begin{document}

\title{Effective Heat Transfer Between a Porous Medium and a Fluid Layer: Homogenization and Simulation}
\author[$\dagger$]{Michael Eden}
\author[$\star$]{Tom Freudenberg}

\affil[$\dagger$]{Department of Mathematics and Computer Science, Karlstad University, Sweden}
\affil[$\star$]{Center for Industrial Mathematics, University of Bremen, Germany}

\maketitle
\begin{abstract}
We investigate the effective heat transfer in complex systems involving porous media and surrounding fluid layers in the context of mathematical homogenization.
We differentiate between two fundamentally different cases: Case (a), where the solid part of the porous media consists of disconnected inclusions, and Case (b), where the solid matrix is connected.
For both scenarios, we consider a heat equation with convection where a small scale parameter $\e>0$ characterizes the heterogeneity of the porous medium and conduct a limit process $\e\to0$ via two-scale convergence for the solutions of the $\e$-problems.
In Case (a), we arrive at a one-temperature problem exhibiting a memory term and, in Case (b), at a two-phase mixture model.
We compare and discuss these two limit models with several simulation studies both with and without convection.
\end{abstract}



\section{Introduction}
The effective transfer of (heat) energy at interfaces between fluid-saturated porous media and adjacent fluid layers plays an important role in many applications; drying processes, metalworking with cutting fluids, geothermal engineering, or transpiration cooling to name just a few \cite{AV01}.
The exact conditions usually depend not just on the specific application but also on the precise geometric setup of the porous medium.
One specific application motivating our research is the impact of cooling fluids during grinding processes.
Some grinding wheels exhibit a porous \me{structure} due to the binding material while the mechanical forces during grinding processes lead to a substantial heat production.
But, as there is a constant resupply of cooling fluid, \emph{local thermal equilibrium} (\cite{W99}) between the fluid inside the porous media and the solid matrix may not be maintained.
\me{This means that there could be a temperature difference across the solid-fluid interface}.
In addition, the precise transfer conditions for the heat across the interface between porous media and fluid layer remain unclear.
\menew{A similar scenario for non-porous grinding wheels with heat resistivity was investigated in \cite{GuShillor2004}; and an extensive, comparative assessment of different approaches of modeling the heat dynamics in surface grinding can be found in \cite{Yang2023}.
In \cite{Heinzel2020}, the authors look at the general interactions between cooling fluids and grinding wheels, also in the case of porous grinding wheels, but without looking at heat dynamics.}
Some early experimental investigations and corresponding simulations for a grinding scenario that takes into account the porous bonding structure of the grinding wheel can be found in \cite{WBE22}.

From an engineering, more applied perspective, this search for effective transfer conditions is often conducted via volume averaging techniques \cite{W99}.
Regarding the specific case of effective heat transfer conditions, the seminal works of Ochoa-Tapia and Whitaker \cite{OT98, OT98b}, where flux conditions were formally derived under the assumption of continuity of temperature and velocity across the interface, stands out; though, as is often the case, there were some earlier contributions as well, e.g., \cite{P90,SK94,VT87}.
With similar approaches several different models and conditions were derived and discussed, cf.~\cite{AC11,AA21,JC09,N12,YV11}.
However, there still seems to be some disagreement about \tfnew{the circumstances under which} those models are valid \cite{N12}.

A different strategy is usually chosen in the context of mathematical homogenization, where limits are investigated with respect to a scale parameter $\e$.
Two-temperature models describing the dispersive and convective transport in porous media \me{(without adjacent fluid layer and in a stationary setting)} were derived in \cite{GP17}.
Closely related, thermo-elasticity problems for different two-phase structures were derived in \cite{EM17,EH15} \me{and the effective heat conductivity for two-phase composites with heat resistance was investigated in \cite{LV96}.
These works do not consider an adjacent fluid layer and do not include convection.}
Similar geometric setups \me{(porous media with adjacent fluid layer)} can be found in fluid scenarios (without heat) involving an interface between porous media and fluid layers, e.g., \cite{ER21,JM00,MM12}.
To our knowledge, there does not yet exist a rigorous investigation of the effective heat transfer for porous media in contact with a fluid layer.
\me{In this work, we are deriving effective models for two different geometrical setups and are able to recover the models presented in \cite{OT98,YV11}.
Consequently, our work gives theoretical support to those models.}

In this work, we start with a two-domain heat equation with imperfect heat transfer between fluid and solid phase and with a prescribed fluid velocity.
By conducting a limit analysis with respect to a scale parameter $\e$, which represents the size of the porous media's microstructures, we identify effective model descriptions.
This is done for two different sets of microstructures (see also \cref{fig:geometry}):
\begin{itemize}
    \item \textbf{Case (a)}: Disconnected solid parts \menew{are periodically distributed in the fluid}. Here, we arrive at a two-scale problem where the solid takes the role of distributed microstructures similarly as in \cite{EM22,SV04} (\cref{hom:a}).
We show that the solid temperature can be eliminated from this model via a memory term, leading to a non-local in time parabolic limit system (\cref{hom:b2}).
For examples with comparable memory terms in the context of porous media, we point to \cite{A13,A92,BLM96,EB14,PSY15}.

\me{Our limit problem in Case (a) reduces to the system given in \cite[Theorem 6]{Pol15} when we do not consider any fluid flow, remove the adjacent fluid layer, and go to a stationary setting.}
    \item \textbf{Case (b)}: Both the pore space and the solid matrix are connected. Here, the limit problem is structurally a two-phase mixture model similar to \cite{GPS14,YV11} (\cref{hom:b}).

    \me{In the absence of an additional fluid layer and considering a stationary setup, the two-phase mixture model in Case (b) simplifies to the models presented in \cite{GP17, RP05}. Additionally, the limit system for the disconnected case coincides with the model derived through RVE averaging in \cite[Eqs. (65c-e)]{OT98}.
It is worth noting that the RVE averaging technique lacks mathematical rigor and relies on challenging-to-verify assumptions such as \emph{local gradient equilibrium} \cite{N12, OT98}.
In this context, \cref{hom:b} provides mathematical justification and support for the averaging results obtained in \cite{OT98}.
Furthermore, our model aligns with the 1D model proposed in \cite{YV11}.}
\end{itemize}

\me{The case of perfect heat transfer (i.e., continuity of temperature across the solid-fluid interface) has also been studied in the literature via RVE averaging, e.g., \cite{OT98b}.
See also \cref{rem:perf_trans} where we discuss this case further and point to the expected changes in the limit problem in this case.}

Numerical investigations for similar limit problems can be found in the literature, e.g., \cite{AC11,AV01,P90}.
In our finite element simulations (using FEniCS), we experiment with different parameters (e.g., heat conductivity, permeability) to compare the two different models both with and without convection.
In particular, we are able to highlight the transition from Case (b) to Case (a) for geometries with bad connectivity (see \cref{ssec:transition}).

This paper is structured as follows: In \cref{sec:setup}, we introduce the mathematical model and the two different geometric setups. 
This is followed by \cref{sec:main_results}, where we present the main results, in particular the homogenization limits for the disconnected case (\cref{hom:a}) and the connected case (\cref{hom:b}).
The detailed proof of these limits via two-scale convergence is presented in \cref{sec:homogenization}.
Finally, in \cref{sec:simulations}, several simulation results are used to compare the different homogenization limits in cases with and without convection.

\section{Setup, mathematical model, and assumptions}\label{sec:setup}
In this section, we provide the two different geometric setups as well as the mathematical model we are considering in this work.
We also collect the assumptions on the coefficients and data. 
Please note, that in the following, we take superscripts $f$ and $s$ to denote domains, functions, and coefficients related to the fluid and solid domain, respectively.
In addition, the superscript $ff$ denotes anything purely related to the \emph{free fluid}, i.e., the adjacent fluid layer, \menew{and the superscript $p$ to the porous medium (see \cref{fig:geometry}).}

First, let $S=(0,T)$, $T>0$, represent the time interval of interest.
For some $H>0$, let $\Omega=\tilde{\Omega}\times(0,H)\subset\R^d$ ($d=2,3$) where $\tilde{\Omega}\subset\R^{d-1}$ is a bounded Lipschitz domain.
\me{In other words, $\Omega$ is a cylinder of $\tilde{\Omega}$.}
We assume that $\Omega$ is a finite union of axis-parallel cubes with corner coordinates in $\Z^{d}$.
\menew{This technical mathematical assumption ensures the existence of $\e$-uniform Sobolev extension operators, see \cref{ext_op}.
Without this assumptions, such operators might not exist \cite{Acerbi1992}.}
\me{We denote the unit normal vector of $\partial\Omega$ pointing outwards of $\Omega$ by $\nu=\nu(x)$ for $x\in\partial\Omega$}.
We subdivide $\Omega$ into subdomains $\Omega^p=\tilde{\Omega}\times(0,h)$ and $\Omega^{ff}=\tilde{\Omega}\times(h,H)$ (for some $0<h<H$) representing the porous domain and the domain of free-flowing fluid.
The interface between these subdomains is denoted by $\Sigma=\tilde{\Omega}\times\{h\}$ \me{and its unit normal vector pointing outwards $\Omega^{ff}$ by $n_\Sigma=n_\Sigma(x)$, $x\in\Sigma$}.

Now, for the reference geometry, let $Y=(0,1)^d$.
Take $Y^{f},$ $Y^{s}\subset Y$ to be two disjointed Lipschitz domains such that $Y=Y^{f}\cup Y^{s}\cup \Gamma$ where $\Gamma:=\overline{Y^{f}}\cap\overline{Y^{s}}$.
Also, let $\Sigma_0=(0,1)^{d-1}\times\{0\}$ denote the lower face of $Y$.
Let $\e_0>0$ be chosen such that \menew{both $\Omega^p$ and $\Omega^{ff}$} can be perfectly tiled with $\e_0Y$ cells and set $\e_n=2^{-n}\e_0$.
We introduce the periodic structures (for the sake of readability we suppress the subscript $n$ and just write $\e$)
    \[
    Z_\e^r=\interior\left(\bigcup_{k\in\Z^d}\e(\overline{Y^{r}}+k)\right)\quad (r=s,f)
    \]
and consider two distinct specific cases:\footnote{There are of course scenarios that are not covered by either case like $Y^s$ being an open ball touching the external boundary of $Y$.}
\begin{itemize}
    \item Case (a): We assume $\overline{Y^s}\subset Y$. As a consequence, $Z_\e^s$ is disconnected.
    \item Case (b): $Z_\e^r$ are Lipschitz domains for both phases $r=f,s$.
    In particular, both sets are connected.\footnote{Please note that this setup is not possible for $d=2$.}
\end{itemize}


\begin{figure}
    \centering
    \includegraphics[width=\linewidth]{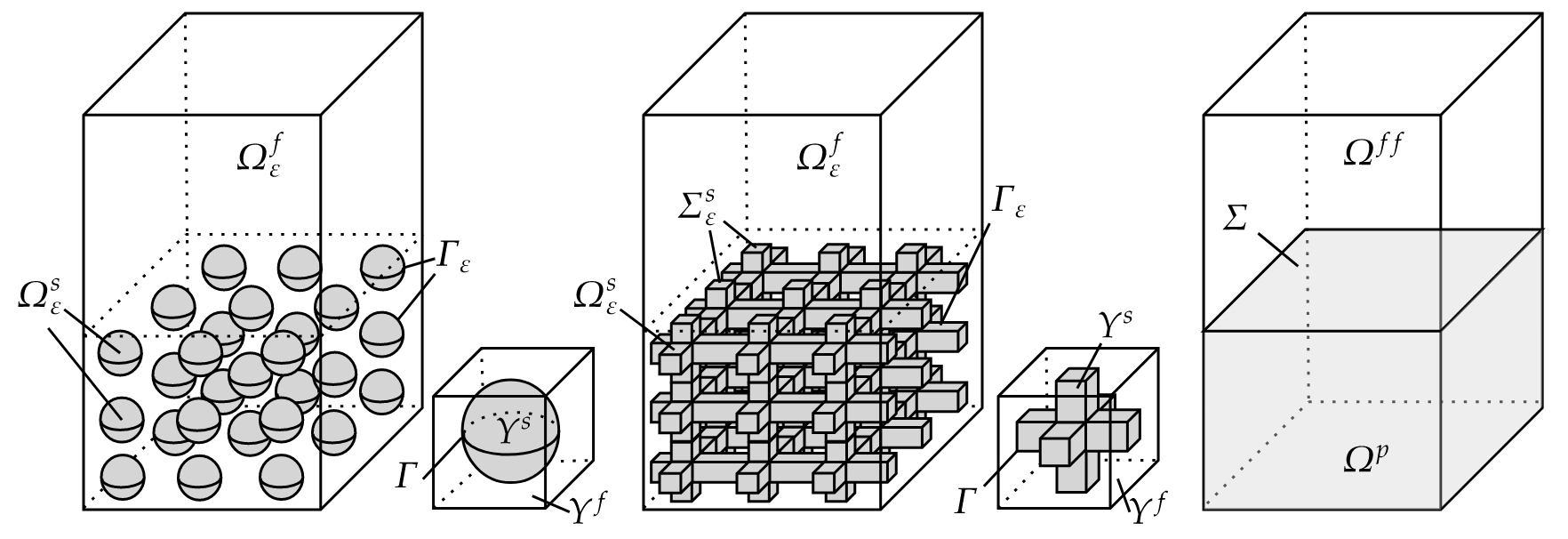}
    \caption{\tf{Example geometry $\Omega$ for the considered problems. Left: the disconnected case with spherical pores and corresponding unit cell \menew{(Case (a))}. Middle: domain for the connected pore structure and unit cell \menew{(Case (b))}. Right: domain in the homogenized model.}}
    \label{fig:geometry}
\end{figure}

In both cases, the \me{unit normal vector} of $\Gamma$ pointing outwards of $Y^{s}$ will be denoted with $n_\Gamma=n_\Gamma(y)$, $y\in\Gamma$.
Now, for $\e>0$, we introduce the $\e Y$-periodic domains $\Omega_\e^{f}$, $\Omega_\e^{s}\subset\Omega^p$ and the interface $\Gamma_\e$ representing the fluid and solid parts of the porous domain and their internal boundary, respectively:
\begin{align*}
	\Omega_\e^{f}=\interior\left(\overline{\Omega^{ff}}\cup\left(\Omega^p\cap Z_\e^f\right)\right),\quad
    \Omega_\e^{s}=\Omega^p\cap Z_\e^s,\quad
	\Gamma_\e=\Omega^p\cap\left(\bigcup_{k\in\Z^d}\e(\overline{\Gamma}+k)\right).
\end{align*}
In Case (a), $\Omega_\e^f$ is connected and $\Omega_\e^s$ is disconnected; in Case (b) both phases are connected.
The \me{unit} normal vector of $\Gamma_\e$ pointing outwards $\Omega_\e^s$ is given by $n_\e=n_\e(x)$, $x\in\Gamma_\e$.
We also introduce the sets
\[
\Sigma_\e^r=\overline{\Omega_\e^r}\cap\Sigma,\quad (\partial\Omega)_\e^r=\me{\partial\Omega_\e^r}\cap\partial\Omega
\]
and note that in Case (a) $\Sigma_\e^s=(\partial\Omega)_\e^s=\emptyset$.
In the following, $\chi_\e^{r}\colon\Omega\to\{0,1\}$ and $\chi^{r}\colon Y\to\{0,1\}$ denote the characteristic functions corresponding to \me{$\Omega_\e^r$} and $Y^r$, respectively.

Now, regarding the mathematical model, let $\theta_\e^r=\theta_\e^r(t,x)$ ($r=f,s$) represent the temperature at time $t\in S$ at $x\in\Omega_\e^r$.
The standard linear heat equation models the heat dynamics in fluid and solid regions via
\begin{subequations}\label{problem1}
\begin{alignat}{2}
\rho^fc^f\partial_t\theta_\e^f-\dive(\kappa^f\nabla\theta_\e^f-\rho^fc^fv_\e\theta_\e^f)&=f_\e^f&\quad&\text{in}\ \ S\times\Omega_\e^f,\label{problem1:1}\\
\rho^sc^s\partial_t\theta_\e^s-\dive(\kappa^s_\e\nabla\theta_\e^s)&=f_\e^s&\quad&\text{in}\ \ S\times\Omega_\e^s.\label{problem1:2}
\end{alignat}
Here, $\rho^r$ denotes the mass density of phase $r$, $c^r$ the specific heat, $\kappa^f$, $\kappa_\e^s$ the heat conductivities, and $f_\e^r$ volume source densities.
Moreover, $v_\e$ denotes the fluid velocity which is assumed to be \menew{known} but has to satisfy the \me{no--slip condition $v_\e=0$} on $\Gamma_\e$, i.e., no inflow of fluid into the solid structure is possible.
In a general setting, the velocity $v_\e$ should ideally be given via the solution to a Navier--Stokes or Stokes system.
In our setting, we just assume the velocity to be given and to satisfy certain assumptions that we specify later.

At the fluid-solid interface $\Gamma_\e$, we assume continuity of fluxes (energy balance) and heat exchange via temperature difference (\menew{note that} $v_\e=0$ on $\Gamma_\e$)
\begin{alignat}{2}
-\kappa^f\nabla\theta_\e^f\cdot n_\e&=-\kappa_\e^s\nabla\theta_\e^s\cdot n_\e&\quad&\text{on}\ \ S\times\Gamma_\e,\label{problem1:3}\\
\kappa^f\nabla\theta_\e^f\cdot n_\e&=\alpha_\e(\theta_\e^f-\theta_\e^s)&\quad&\text{on}\ \ S\times\Gamma_\e.\label{problem1:4}
\end{alignat}
Here, $\alpha_\e$ denotes the heat exchange coefficient.
For small values of $\alpha_\e$ (relative to $|\Gamma_\e|^{-1}$), the thermal resistivity condition given by \cref{problem1:4} approximates thermal isolation (no heat exchange between the subdomains, i.e., $\kappa^f\nabla\theta_\e^f\cdot n_\e=0$).
Conversely, for large values of $\alpha_\e$, the condition approximates temperature continuity across the interface (perfect heat exchange, i.e., $\theta_\e^f=\theta_\e^s$).
The latter case, where temperature continuity is maintained, is often referred to as local thermal equilibrium (LTE) as discussed in \cite{W99}.
This assumption is reasonable for many applications, especially in the context of steady conduction problems \cite{RP05}.
For a comprehensive overview and critical evaluation of the LTE assumption and the corresponding thermal resistance approach, which is also known as local thermal non-equilibrium (NLTE) and is utilized in this work, we refer to \cite{Pati22} and \cite[Section 6.3]{RP05} and the references therein.
In general, the thermal resistance model is more versatile as it can be used to approximate thermal equilibrium for large values of $\alpha_\e$.

One specific application that motivates the use of this thermal resistivity model (we also refer to \cite{WBE22}), is the interaction between cooling fluids and grinding wheels.
In this scenario, thermal equilibrium is not typically expected due to (a) differing thermal properties between the fast-flowing cooling fluid (often oils) and the complex composition of the grinding wheel (a composite of polymers, metals, and diamond grains) and (b) the continuous influx of cooler fluid compared to the porous medium.
\menew{Naturally, our model is still an idealization of this complex real world process.
For example, we do not consider the workpiece and we are also neglecting the abrasive grains of the grinding wheel.
Moreover, chip formation and transport, which also play a role in the heat transport in grinding processes, are not considered.
This is because we are specifically interested in investigating the interplay of the porous bonding structure with the heat dynamics in this work.
However, please note, the model is not restricted to this specific application and the results transfer to other applications like geothermal engineering; see \cite{OT98,YV11} where similar models are considered.}

Finally, we pose homogeneous Neumann boundary conditions at the external boundaries \menew{($\partial\Omega)_\e^r$} and initial conditions ($r=f,s$):
\begin{alignat}{2}
-\kappa^r\nabla\theta_\e^r\cdot \nu&=0&\quad&\text{on}\ \ S\times(\partial\Omega)_\e^r,\label{problem1:5}\\
\theta_\e^r(0, \cdot)&=\theta_{\e,0}^r&\quad&\text{in}\ \ \Omega_\e^r.
\end{alignat}
Please note that in Case (a), condition \eqref{problem1:5} is vacuously true for the solid part as $(\partial\Omega)_\e^s=\emptyset$ and $(\partial\Omega)_\e^f=\partial\Omega$, but in Case (b) both conditions are needed.
The $\e$-dependent micro-model in its PDE form is then given by system \eqref{problem1}.
\end{subequations}

In the following, we collect the main assumptions we are placing on the coefficients and data regarding the problem. 
The precise definition of a solution is given in \cref{sec:main_results}.
In general, for a function $f$ defined on $\Omega_\e^f$ or $\Omega_\e^s$, $\tilde{f}$ denotes the zero-continuation to the whole of $\Omega$ or $\Omega^p$, respectively.
Also, the subscript $\#$ in function spaces is taken to indicate periodicity, e.g.,
\[
H^1_\#(Y):=\left\{u\in H^1_{loc}(\R^d) \ :\ u_{|Y}\in H^1(Y),\  \ u(x+e_i)=u(x)\ \text{for almost all}\ x\in \R^d,\ i=1,2,3\right\}.
\]
\menew{For the limit analysis, we rely on the concept of two-scale convergence, see, e.g.,~\cite{A92,LNW02} for an overview:
A sequence $v_\e\in L^2(S\times \Omega)$ is said to two-scale converge to a function $v\in L^2(S\times\Omega\times Y)$ for $\e\to0$ (notation $v_\e\twosc v$) if 
\begin{equation}\label{twoscale}
\lim_{\e\to0}\int_S\int_\Omega v_\e(t,x)\varphi(t,x,\nicefrac{x}{\e})\di{x}\di{t}=\int_S\int_\Omega\int_Yv(t,x,y)\varphi(t,x,y)\di{y}\di{x}\di{t}
\end{equation}
for all test functions $\varphi\in L^2(S\times\Omega;C_\#(Y))$.}

\paragraph{Assumptions on the data.}

\begin{itemize}
    \item[(A1)] The coefficients $\rho^r$, $c^r$, and  $\kappa^f$ are positive.
    Also, there are positive constants $\alpha$, $\kappa^s$ such that
    \begin{itemize}
        \item[(a)] \textbf{$\Omega_\e^s$ disconnected:} $\kappa_\e^s=\e^2\kappa^s$, $\alpha_\e=\e\alpha$
        \item[(b)] \textbf{$\Omega_\e^s$ connected:}  $\kappa_\e^s=\kappa^s$, $\alpha_\e=\e\alpha\chi_{\Gamma_\e\cap\Omega^p}+\alpha\chi_{\Sigma_\e^s}$
    \end{itemize}
    Please note that the $\e^2$-scaling in Case (a) is standard for this particular type of the disconnected solid geometry, cf.~\cite{ADH90,EB14,GNP22,HJ91} for comparable situations.
    The heat exchange coefficient $\alpha$ is scaled with $\e$ to counteract \menew{the blow-up of the size of the interface, i.e., $\lim_{\e\to0}|\Gamma_\e|=\infty$, via $\e|\Gamma_\e|=|\Omega^p||\Gamma|$}; as is common with this kind of interface terms \cite{HJ91,N96,PS08}.
    In the connected case, the additional interface $\Sigma_\e^s$ does not blow-up \menew{($\lim_{\e\to0}|\Sigma_\e^s|=|\Sigma||\Sigma^s|$)} and for that reason is not scaled. 
    \item[(A2)] The volume source densities $(f_\e^f,f_\e^s)\in L^2(S\times\Omega_\e^f)\times L^2(S\times\Omega_\e^s)$ satisfy
    \[
    C_f:=\sup_{\e>0}\left(\|f_\e^f\|_{L^2(S\times\Omega_\e^f)}^2+\|f_\e^s\|^2_{L^2(S\times\Omega_\e^s)}\right)<\infty.
    \]
    \item[(A3)] The initial conditions $(\theta_{\e,0}^f,\theta_{\e,0}^s)\in L^2(\Omega_\e^f)\times L^2(\Omega_\e^s)$ satisfy 
    \[
    C_0:=\sup_{\e>0}\left(\|\theta_{\e,0}^f\|_{L^2(\Omega_\e^f)}^2+\|\theta_{\e,0}^s\|^2_{L^2(\Omega_\e^s)}\right)<\infty.
    \]
    \item[(A4)] The velocity $v_\e\in L^\infty(S\times\Omega_\e^f)^d$ satisfies
    \[
    C_v:=\sup_{\e>0}\|v_\e\|_{L^\infty(S\times\Omega_\e^f)}<\infty.
    \]
    Also, we assume the fluid to be incompressible\menew{, i.e., $\dive v_\e=0$,} as well as \me{$v_\e=0$} on $\Gamma_\e$.
    \item[(A5)] There are limit functions $f^f\in L^2(S\times\Omega)$ as well as $\theta_0^f\in L^2(S\times\Omega)$ such that
    \[
    \tilde{f}_\e^f\to f^f,\quad \tilde{\theta}_\e^f\to \theta_0^f \quad\text{in}\ L^2(S\times\Omega) \tfnew{\text{, for } \varepsilon \to 0.}
    \]
    \item[(A6)] 
    \begin{itemize}
        \item[(a)] \textbf{$\Omega_\e^s$ disconnected.} There are functions $f^s\in L^2(S\times\Omega^p\times Y)$ as well as $\theta_0^s\in L^2(\Omega^p\times Y)$ such that
        \[
        f_\e^s\twosc\chi^sf^s,\quad \theta_{\e,0}^s\twosc\chi^s\theta_0^s, \quad\tfnew{\text{ for } \varepsilon \to 0.}
        \]
        \item[(b)] \textbf{$\Omega_\e^s$ connected.}  There are functions $f^s\in L^2(S\times\Omega^p)$ as well as $\theta_0^s\in L^2(\Omega^p)$ such that
        \[
        f_\e^s\to f^s,\quad \theta_{\e,0}^s\to \theta_0^s\quad \text{in}\ L^2(S\times\Omega) \tfnew{\text{, for } \varepsilon \to 0.}
        \]
    \end{itemize}
    \item[(A7)] There are functions $v\in L^2(S\times\Omega^{ff})$, $v_D\in L^2(S\times\Omega^p;H_\#^1(Y))$ where $v_D(t,\cdot,y)\in H^1(\Omega^p)$ for almost all $(t,y)\in S\times Y$ such that
    \[
    v_\e\to v \quad\text{in}\ L^2(S\times\Omega^{ff}),\quad{\tilde{v}_\e}{}_{|\Omega^p}\twosc\chi^fv_D, \quad\tfnew{\text{ for } \varepsilon \to 0.}
    \]
    We also assume that \me{$v_D=0$} almost everywhere on $S\times\Omega\times\Gamma$ and incompressibility, i.e., $\dive_yv_D=0$ almost everywhere in $S\times\Omega\times Y^f$.
    \me{Finally, we impose $\int_{Y^f}v_D\di{y}\cdot n_\Sigma=v\cdot n_\Sigma$.}
\end{itemize}
From this list, Assumptions (A1)--(A4) are needed to ensure the existence of unique solutions with certain $\e$-uniform estimates and Assumptions (A5)-(A7) for the limit process $\e\to0$.
    The letters (a) and (b) indicate the specific geometric setup.
\me{Please note that assumptions similar to $(A4)$ and $(A7)$ are often posed in homogenization scenarios including convection/advection, e.g., \cite{AllaireHabibi,EM17,GNP22}.
More concretely, it is usually assumed that the velocity can be represented as $v_\e(x)=v(x,\nicefrac{x}{\e})$ for some continuous function $v$ which is $Y$-periodic in its second argument.}
This setup is also used in \emph{fast-drift} problems \cite{Allaire10,AMP10}.
\menew{The continuity of normal velocities, the last condition in Assumption (A7), is physically motivated by the principle of mass balance.}
Although this condition is not necessary from a mathematical point of view, \menew{it is consistent} with the usual interface conditions for fluid systems (e.g., Joseph--Beavers and its generalizations).


\section{Main results}\label{sec:main_results}
We fix our solution space
\[
W_\e=\left\{(u^f,u^s)\in L^2(S; H^1(\Omega_\e^f)\times H^1(\Omega_\e^s))\ : \ \partial_t(u^f,u^s)\in L^2(S;L^2(\Omega_\e^f\times \Omega_\e^s))\right\}
\]
and call $(\theta_\e^f,\theta_\e^s)\in W_\e$ a solution to the $\e$-dependent problem given by system \eqref{problem1} if it satisfies $(\theta_\e^f(0),\theta_\e^s(0))=(\theta_{\e,0}^f,\theta_{\e,0}^s)$ as well as
\begin{multline}\label{eps_weak_form}
(\rho^fc^f\partial_t\theta_\e^f,\varphi^f)_{\Omega_\e^f}+(\kappa^f\nabla\theta^f_\e-\rho^fc^fv_\e\theta^f_\e,\nabla\varphi^f)_{\Omega_\e^f}
+(\rho^sc^s\partial_t\theta_\e^s,\varphi^s)_{\Omega_\e^s}+\kappa_\e^s(\nabla\theta^s_\e,\nabla\varphi^s)_{\Omega_\e^s}\\
+\alpha_\e(\theta_\e^f-\theta_\e^s,\varphi^f-\varphi^s)_{\Gamma_\e}=(f_\e^f,\varphi^f)_{\Omega_\e^f}+(f_\e^s,\varphi^s)_{\Omega_\e^s}
\end{multline}
for all test functions $(\varphi^f,\varphi^s)\in H^1(\Omega_\e^f)\times H^1(\Omega_\e^s)$ and almost all $t\in S$.

\begin{theorem}[Existence and estimates]\label{ex_est}
    Let Assumptions (A1)--(A4) be satisfied.
    Then, there is a unique $(\theta_\e^f,\theta_\e^s)\in W_\e$ satisfying $(\theta_\e^f(0),\theta_\e^s(0))=(\theta_{\e,0}^f,\theta_{\e,0}^s)$ and \cref{eps_weak_form} for all test functions $(\varphi^f,\varphi^s)\in H^1(\Omega_\e^f)\times H^1(\Omega_\e^s)$ and almost all $t\in S$.
    In addition, \tfnew{there exists a constant $C>0$ independent of $\varepsilon$ such that}
    \begin{equation}\label{esp_estimate}
    \|\theta_\e^f\|^2_{L^\infty(S;L^2(\Omega_\e^f))}+\|\theta_\e^s\|^2_{L^\infty(S;L^2(\Omega_\e^s))}
    +\|\nabla\theta_\e^f\|^2_{L^2(S\times \Omega_\e^f)}+\e^\gamma\|\nabla\theta_\e^s\|^2_{L^2(S\times \Omega_\e^s)}
    +\e\|\theta_\e^f-\theta_\e^s\|^2_{L^2(S\times\Gamma_\e)}\leq C
    \end{equation}  
    where $\gamma=2$ in Case (a) and $\gamma=0$ in Case (b).
\end{theorem}

\begin{proof}
    For each $\e>0$, \cref{eps_weak_form} is the weak form of a linear advection-diffusion parabolic system which admits a unique solution $(\theta_\e^f,\theta_\e^s)\in W_\e$ under the given assumptions.
    The estimates are the result of energy estimates:
    Taking $(\theta_\e^f,\theta_\e^s)$ as a test function, we get
    \begin{multline*}
    \rho^fc^f\ddt\|\theta_\e^f\|^2_{L^2(\Omega_\e^f)}+\kappa^f\|\nabla\theta^f_\e\|^2_{L^2(\Omega_\e^f)}
    +\rho^sc^s\ddt\|\theta_\e^s\|^2_{L^2(\Omega_\e^s)}+\e^\gamma \kappa^s\|\nabla\theta^s_\e\|^2_{L^2(\Omega_\e^s)}
    +\e\alpha\|\theta_\e^f-\theta_\e^s\|^2_{L^2(\Gamma_\e)}\\
    \leq(f_\e^f,\theta_\e^f)_{\Omega_\e^f}+(f_\e^s,\theta_\e^s)_{\Omega_\e^s}
    +\rho^fc^f(v_\e\theta^f_\e,\nabla\theta_\e^f)_{\Omega_\e^f}.
    \end{multline*}
    Integrating over $(0,t)$ and using $\menew{\|v_\e\|_{L^\infty(S\times\Omega)}}\leq C_v$ (Assumption (A4)) yields
    \begin{multline*}
    \rho^fc^f\|\theta_\e^f(t)\|^2_{L^2(\Omega_\e^f)}+\rho^sc^s\|\theta_\e^s(t)\|^2_{L^2(\Omega_\e^s)}\\
    +\frac{\kappa^f}{2}\int_0^t\|\nabla\theta^f_\e\|^2_{L^2(\Omega_\e^f)}\di{\tau}+\e^\gamma\kappa^s\int_0^t\|\nabla\theta^s_\e\|^2_{L^2(\Omega_\e^s)}\di{\tau}
    +\e\alpha\int_0^t\|\theta_\e^f-\theta_\e^s\|^2_{L^2(\Gamma_\e)}\di{\tau}\\
    \leq\int_0^t(f_\e^f,\theta_\e^f)_{\Omega_\e^f}\di{\tau}+\int_0^t(f_\e^s,\theta_\e^s)_{\Omega_\e^s}\di{\tau}
    +\frac{(C_v\rho^fc^f)^2}{2\kappa^f}\int_0^t\|\theta^f_\e\|^2\di{\tau}
    +\rho^fc^f \|\theta_{\e,0}^f\|^2_{L^2(\Omega_\e^f)}+\rho^sc^s\|\theta_{\e,0}^s\|^2_{L^2(\Omega_\e^s)}.
    \end{multline*}
    Consequently, using Gronwall's inequality, we arrive at estimate \eqref{esp_estimate} where $C>0$ is independent of $\e$.
\end{proof}

\begin{theorem}[Homogenization (Case a)]\label{hom:a}
    Let Assumptions (A1)--(A7) (in their (a)-Variants) be satisfied.
    Then, $\theta_\e^f\to\Theta$ in $L^2(S\times\Omega)$ and $\theta_\e^s\twosc\theta^s$ in $L^2(S\times\Omega^p\times Y^s)$ for $\e\to0$, where
    \[
    (\Theta,\theta^s)\in L^2(S;H^1(\Omega))\times L^2(S\times\Omega;H^1(Y^s))\ \text{s.t.}\ 
    (\partial_t\Theta,\partial_t\theta^s)\in L^2(S\times\Omega)\times L^2(S\times\Omega\times Y^s)
    \]
    is the unique solution of the homogenized fluid-heat system \me{(we set $\theta^{ff}=\Theta_{|\Omega^{ff}}$ and $\theta^{f}=\Theta_{|\Omega^{p}}$)}
    \begin{subequations}\label{homa}
    \begin{alignat}{2}
    \rho^fc^f\partial_t\theta^{ff}-\dive(\kappa^f\nabla\theta^{ff}-\rho^fc^fv\theta^{ff})&=f^f&\quad&\text{in}\ \ S\times\Omega^{ff},\label{thm_d:hom1}\\
    |Y^f|\rho^fc^f\partial_t\theta^f-\dive(\kappa_h^f\nabla\theta^f-\rho^fc^f\bar{v}_D\theta^f)+\alpha\int_\Gamma(\theta^{f}-\theta^s)\di{\sigma}&=|Y^f|f^f&\quad&\text{in}\ \ S\times\Omega^p,\label{thm_d:hom1b}\\
    \me{(\kappa^f\nabla\theta^{ff}-\kappa_h^f\nabla\theta^f)\cdot n_\Sigma}&=0&\quad&\text{on}\ \ S\times\Sigma,\label{thm_d:hom3c}\\
    \me{\theta^f}&\me{=\theta^{ff}}&\quad&\text{\me{on}}\ \ \me{S\times\Sigma},\label{thm_d:hom3c2}\\
    -\kappa^f\nabla\theta^{ff}\cdot\nu&=0&\quad&\text{on}\ \ S\times(\partial\Omega\cap\partial\Omega^{ff}),\label{thm_d:hom3}\\
    -\kappa_h^f\nabla\theta^f\cdot\nu&=0&\quad&\text{on}\ \ S\times(\partial\Omega\cap\partial\Omega^{p}),\label{thm_d:hom3b}\\
    \tf{\Theta(0, \cdot)}&=\theta^f_0&\quad&\text{in}\ \ \Omega.\label{thm_d:hom5}
    \end{alignat}
    \me{It is} coupled with the microscale solid-heat problem
    \begin{alignat}{2}
    \rho^s c^s \partial_t\theta^s-\kappa^s\Delta_y\theta^s&=f^s&\quad&\text{in}\ \ S\times\Omega^{p}\times Y^s,\label{thm_d:hom2}\\
    -\kappa^s\nabla\theta^s\cdot n_\Gamma&=\alpha(\theta^s-\theta^f)&\quad&\text{on}\ \ S\times\Omega^{p}\times\Gamma,\label{thm_d:hom4}\\
    \tf{\theta^s(0, \cdot)}&=\theta_0^{s}&\quad&\text{in}\ \ \Omega^p\times Y_s.\label{thm_d:hom6}
    \end{alignat}    
    The velocity field $\bar{v}_D$ and the heat conductivity matrix $\kappa_h^f$ are given by
    \[
    \bar{v}_D(t,x)=\int_{Y^f}v_D(t,x,y)\di{y},\quad(\kappa_h^f)_{ij}=\me{\kappa^f}\int_{Y^f}(\nabla_y\xi_i^f+e_i)\cdot e_j\di{y}.
    \]
    Here, the $\xi_i^f\in H^1_\#(Y^f)$ ($i=1,2,3$) are the unique, zero-average solution to the cell problems
    \begin{alignat}{2}
    -\Delta_y\xi_i^f&=0&\quad&\text{in}\ \ Y^f,\label{thm_d:hom7}\\
    -\nabla_y\xi_i^f\cdot n_\Gamma&=e_i\cdot n_\Gamma&\quad&\text{on}\ \ \Gamma,\label{thm_d:hom8}\\
    y&\mapsto\xi_i^f(y)&\quad&\text{is $Y$-periodic}\label{thm_d:hom9}.
    \end{alignat}
    \end{subequations}
\end{theorem}
\begin{proof}
\menew{The general strategy of the homogenization process is outlined in \cref{rem:strategy} of \cref{sec:homogenization} and the individual steps are presented in detail in \cref{sec:disc}.
More specifically,
\begin{itemize}
    \item in \cref{lem:limits}, the existence of the limit functions, at least for subsequences, is shown,
    \item a limit analysis $\e\to0$ leads to the coupled system given by \cref{d:hom_var1,d:hom_var2,d:hom_var3,d:hom_var4},
    \item and the subsequent decoupling process results in system \ref{homa}.
\end{itemize}}

For uniqueness, let \tfnew{$\{\Theta_j,\theta_j^s\}_{j=1,2}$} be two sets of solution, whose differences we denote by $(\overline{\Theta},\overline{\theta}^s)$.
Standard energy estimations for the differences lead to
\begin{multline*}
\ddt\|\overline{\Theta}\|^2_{L^2(\Omega)}
+\|\nabla\overline{\Theta}\|^2_{L^2(\Omega)}
+|\Gamma|\|\overline{\theta}^f\|_{L^2(\Omega^p)}^2\\
\leq C\left(\|\widehat{v}\|_{L^\infty(S\times\Omega)}\|\nabla\overline{\Theta}\|_{L^2(\Omega)}\|\overline{\Theta}\|_{L^2(\Omega)}+\|\overline{\theta}^s\|_{L^2(\Omega^p\times\Gamma)}\|\overline{\Theta}\|_{L^2(\Omega)}\right)
\end{multline*}
for the fluid temperature and
\begin{equation*}
    \ddt\|\overline{\theta}^s\|^2_{L^2(\Omega^p\times Y^s)}+2\kappa^s\|\nabla\overline{\theta}^s\|^2_{L^2(\Omega^p\times Y^s)}+\alpha\|\overline{\theta}^s\|_{L^2(\Omega^p\times\Gamma)}^2\leq \alpha\sqrt{\Gamma}\|\overline{\theta}^s\|_{L^2(\Omega^p\times\Gamma)}\|\overline{\Theta}\|_{L^2(\Omega)}
\end{equation*}
for the solid temperature.
With the application of Young's and Gronwall's inequalities, it follows that $\overline{\Theta}\equiv0$ and $\overline{\theta}^s\equiv0$ almost everywhere in $S\times\Omega$ and $S\times\Omega\times Y^s$, respectively.
The uniqueness also implies that the whole \tfnew{sequence converges}. 
\end{proof}

\begin{remark}
 Since $\Theta$ is continuous across $\Sigma$ (due to $\Theta\in H^1(\Omega)$), the fluid heat system can also be expressed \menew{as a single PDE} by considering piece-wise constant coefficients, e.g.,  $\widetilde{\rho c}=\chi_{\Omega^{ff}}\rho^fc^f+\chi_{\Omega^{p}}|Y^f|\rho^fc^f$ \menew{and $\widetilde{\alpha}=\chi_{\Omega^{p}}|Y^f|\alpha$}.
This leads to
    \begin{alignat*}{2}
    \widetilde{\rho c}\partial_t\Theta-\dive(\tilde{\kappa}\nabla\Theta-\widetilde{\rho cv}\Theta)+\widetilde{\alpha}\int_\Gamma(\Theta-\theta^s)\di{\sigma}&=\tilde{f^f}&\quad&\text{in}\ \ S\times\Omega,\\
    -\widetilde{\kappa}\nabla\Theta\cdot\nu&=0&\quad&\text{on}\ \ S\times\partial\Omega,\\
    \Theta(0, \cdot)&=\theta^f_0&\quad&\text{in}\ \ \Omega.
    \end{alignat*}
\end{remark}

The linear problem, as described by \cref{thm_d:hom1,thm_d:hom1b,thm_d:hom2,thm_d:hom3,thm_d:hom3b,thm_d:hom3c,thm_d:hom3c2,thm_d:hom4,thm_d:hom5,thm_d:hom6}, represents a classical example of a coupled two-scale model for the porous part $\Omega^p$, where there is a heat exchange between solid and fluid \me{compartments through} the volume source density $\alpha\int_\Gamma(\theta^f-\theta^s)\di{\sigma}$ and the corresponding boundary condition \eqref{thm_d:hom4} for $\theta^s$.
Please note that energy conservation (outside the volume sources $f^f$ and $f^s$) still holds as these contributions balance \me{each other}.
Also, while $\theta^f=\theta^{ff}$ at $\Sigma$, the heat coefficient $\kappa_h$ is generally not continuous across $\Sigma$ (see \cref{ch:stationary_profile}).
As a consequence, $\Theta$ is expected to exhibit \tf{a change in slope across the interface} as can also be seen in the simulations in \cref{sec:simulations}, see \cref{fig:temp_sationary}.

One possible interpretation of the solid inclusions is \me{to view them as} heat sinks or heat sources for the porous medium (depending on the prior history of the system): Due to imperfect heat transfer between fluid and solid parts, there is an expected delay of temperature equilibrium. 
When the porous system cools or heats up, the solid system acts as either a heat source or storage, thus causing a delay.

This idea can be made mathematically explicit by introducing a memory term accounting for the history of the system.
In the resulting model, the solid heat system is decoupled from the fluid system at the cost of two additional cell problems.

\begin{lemma}[Homogenization with memory term]\label{hom:b2}
The system \eqref{thm_d:hom1}--\eqref{thm_d:hom9} can alternatively be written as (plus initial, boundary, and interface conditions)
\begin{subequations}
\begin{alignat}{2}
\rho^fc^f\partial_t\theta^{ff}-\dive(\kappa^f\nabla\theta^{ff}-\rho^fc^fv\theta^{ff})&=f^{ff}&\quad&\text{in}\ \ S\times\Omega^{ff},\label{thm_d:2hom1}\\
\begin{split}
\rho^fc^f|Y^f|\partial_t\theta^f-\dive(\kappa_h^f\nabla\theta^f-\rho^fc^f\bar{v}_D\theta^f)\hspace{3.5cm}&\\
+\alpha|\Gamma|\theta^f-\int_0^t\theta^f(\tau)\psi(t-\tau)\di{\tau}&=|Y^f|f^f+\bar{\eta}\end{split}&\quad&\text{in}\ \ S\times\Omega^p,\label{thm_d:2hom2}
\end{alignat}
where
\[
\psi(t)=\int_\Gamma\partial_t\xi(t,y)\di{\sigma},\quad \bar{\eta}(t,x)=\int_\Gamma\eta(t,x,y)\di{\sigma}+|\Gamma|\theta_0^f(x).
\]
Here, $\xi\in L^2(S;H^1(Y^s))$ with $\partial_t\xi\in L^2(S\times Y^s)$ and $\eta\in L^2(S\times\Omega^p;H^1(Y^s))$ with $\partial_t\eta\in L^2(S\times\Omega^p\times Y^s)$ are the unique solutions of the cell problems

\begin{minipage}[t]{0.45\textwidth}
\begin{alignat*}{2}
\rho^sc^s\partial_t\xi-\kappa^s\Delta_y\xi&=0&\quad&\text{in}\ \ S\times Y^s,\\
-\kappa^s\nabla\xi\cdot n_\Gamma&=\alpha(\xi-1)&\quad&\text{on}\ \ S\times\Gamma,\\
\xi(0, \cdot)&=0&\quad&\text{in}\ \ Y^s,
\end{alignat*}
\end{minipage}\hspace{.5cm}
\begin{minipage}[t]{0.45\textwidth}
\begin{alignat*}{2}
\rho^sc^s\partial_t\eta-\kappa^s\Delta_y\eta&=f^s&\quad&\text{in}\ \ S\times Y^s,\\
-\kappa^s\nabla\eta\cdot n_\Gamma&=0&\quad&\text{on}\ \ S\times\Gamma,\\
\eta(0, \cdot)&=\theta_0^s-\theta_0^f&\quad&\text{in}\ \ Y^s.
\end{alignat*}
\end{minipage}
\end{subequations}
\end{lemma}
\begin{proof}
We take a closer look at the \menew{convolution in time}
\[
\vartheta(t,x,y)=\theta_0^f(x)+\int_0^t\theta^f(\tau,x)\partial_t\xi(t-\tau,y)\di{\tau}
\]
which satisfies $\vartheta(0,x,y)=\theta_0^f(x)$ almost everywhere in $\Omega^p\times Y^s$.
For the regularity, we have
\[
\vartheta\in L^2(S\times\Omega^p; H^1(Y^s))\quad\text{such that}\ \partial_t\vartheta \in L^2(S\times\Omega^p\times Y^s).
\]
We calculate for any test function $\phi\in H^1(Y^s)$ using convolution properties:
\begin{align*}
\int_{Y^s}\partial_t\vartheta(t,x,y)\phi(y)\di{y}
&=\int_{Y^s}\int_0^t\partial_t\theta^f(t-\tau,x)\partial_t\xi(\tau,y)\phi(y)\di{\tau}\di{y},\\
\int_{Y^s}\nabla_y\vartheta(t,x,y)\nabla_y\phi(y)\di{y}&=\int_{Y^s}\int_0^t\partial_t\theta^f(t-\tau,x)\nabla_y\xi(\tau,y)\nabla_y\phi(y)\di{\tau}\di{y},\\
\int_{\Gamma}\alpha(\vartheta(t,x,y)-\theta^f(t,x))\phi(y)\di{\sigma}
&=\int_{\Gamma}\alpha\left(\int_0^t\partial_t\theta^f(t-\tau,x)\xi(\tau,y)\di{\tau}-(\theta^f(t,x)-\theta_0^f(x))\right)\phi(y)\di{\sigma},\\
&=\int_{\Gamma}\int_0^t\partial_t\theta^f(t-\tau,x)\alpha\left(\xi(\tau,y)-1\right)\phi(y)\di{\tau}\di{\sigma}.
\end{align*}
As a result, for almost all $x\in\Omega^p$, $\vartheta(\cdot,x,\cdot)$ is the unique weak solution of
\begin{alignat*}{2}
\rho^sc^s\partial_t\vartheta-\kappa^s\Delta_y\vartheta&=0&\quad&\text{in}\ \ S\times Y^s,\\
-\kappa^s\nabla\vartheta\cdot n_\Gamma&=\alpha(\vartheta-\theta^f)&\quad&\text{on}\ \ S\times\Gamma,\\
\vartheta(0, \cdot)&=\theta_0^f&\quad&\text{in}\ \ Y^s.
\end{alignat*}
Given the linearity of the model, we can eliminate the function $\theta^s$ using the \tfnew{solutions $\vartheta$ and $\eta$ of the cell problems} via
\[
\theta^s=\vartheta+\eta,
\]
which leads to the memory term.
\end{proof}

In this model, the additional source density $\bar{\eta}$ \me{in \cref{thm_d:2hom2}} accounts for the heat transfer \me{between the solid and the fluid system due to both differences in the initial temperature distributions ($\theta_0^s-\theta_0^f$) as well as a result of the source density $f^s$.\medskip}

Considering Case (b), which features a connected solid matrix, we deduce a two-phase mixture model:

\begin{theorem}[Homogenization (Case b)]\label{hom:b}
    Let Assumptions (A1)--(A7) (in their (b)-Variants) be satisfied.
    Then, $\theta_\e^f\to\Theta$ in $L^2(S\times\Omega)$ and $\theta_\e^s\to\theta^s$ in $L^2(S\times\Omega^p)$ for $\e\to0$, where
    \[
    (\Theta,\theta^s)\in L^2(S;H^1(\Omega)\times H^1(\Omega^p))\ \text{s.t.}\ 
    (\partial_t\Theta,\partial_t\theta^s)\in L^2(S\times\Omega)\times L^2(S\times\Omega^p)
    \]
    is the unique solution of the homogenized system \me{(we set $\theta^{ff}=\Theta_{|\Omega^{ff}}$ and $\theta^{f}=\Theta_{|\Omega^{p}}$)}
    \begin{subequations}\label{homb}
    \begingroup
    \allowdisplaybreaks
    \begin{alignat}{2}
    \rho^f c^f \partial_t\theta^{ff}-\dive(\kappa^{f}\nabla\theta^{ff}-\rho^f c^f v\theta^{ff})&=f^{ff}&\quad&\text{in}\ \ S\times\Omega^{ff},\\
    |Y^f|\rho^f c^f\partial_t\theta^{f}-\dive(\kappa_h^f\nabla\theta^{f}-\rho^f c^f\bar{v}_D\theta^{f})&=|Y^f| f^f -\alpha|\Gamma|(\theta^f-\theta^s)&\quad&\text{in}\ \ S\times\Omega^{p},\label{homb_porous}\\
    |Y^s|\rho^s c^s\partial_t\theta^s-\dive(\kappa_h^s\nabla\theta^s)&=|Y^s| f^s + \alpha|\Gamma|(\theta^f-\theta^s)&\quad&\text{in}\ \ S\times\Omega^{p},\\
    \theta^f&=\theta^{ff}&\quad&\text{on}\ \ S\times\Sigma,\\
    -(\kappa_h^s\nabla\theta^s+\kappa_h^f\nabla\theta^{f})\cdot n_\Sigma&=-\kappa^f\nabla\theta^{ff}\cdot n_\Sigma&\quad&\text{on}\ \ S\times\Sigma,\label{homb_porous2}\\
    -\kappa_h^s\nabla\theta^s\cdot n_\Sigma&=\alpha|\Sigma^s|(\theta^s-\theta^f)&\quad&\text{on}\ \ S\times\Sigma\label{homb_porous3}, \\
    \Theta(0, \cdot)&=\theta^{f}_0&\quad&\text{in}\ \ \Omega,\\
    \theta^{s}(0, \cdot)&=\theta^{s}_0&\quad&\text{in}\ \ \Omega^{p}.
    \end{alignat}
    \endgroup
    The matrices $\kappa^f_h,\kappa^s_h\in\R^{3\times3}$ are given by
    \begin{equation*}
    (\kappa_h^f)_{ij}=\me{\kappa^f}\int_{Y^f}(\nabla_y\xi_i^f+e_i)\cdot e_j\di{y},\qquad
    (\kappa_h^s)_{ij}=\me{\kappa^s}\int_{Y^s}(-\nabla_y\xi^s_i+e_i)\cdot e_j\di{y}.
    \end{equation*}
    Here, the $\xi_i^r\in H^1_\#(Y^f)$ ($i=1,2,3$, $r=f,s$) are the unique, zero-average \tfnew{solutions} to the cell problems
    \begin{alignat}{2}
    -\Delta_y\xi_i^r&=0&\quad&\text{in}\ \ Y^r,\\
    -\nabla_y\xi_i^r\cdot n_\Gamma&=e_i\cdot n_\Gamma&\quad&\text{on}\ \ \Gamma,\\
    y&\mapsto\xi_i^r(y)&\quad&\text{is $Y$-periodic}.
    \end{alignat}
    \end{subequations}
\end{theorem}

\begin{proof}
The limits of $(\theta_\e^f,\theta_\e^s)$ and the general limiting procedure are presented in \cref{sec:con}.
The uniqueness follows again from energy estimates similarly to~\cref{hom:b}.
\end{proof}

\begin{remark}[The case of perfect heat transmission]\label{rem:perf_trans}
\menew{Comparing with the homogenization results in \cref{hom:a} for Case (a) and \cref{hom:b} for Case (b), we expect similar homogenization limits when there is no thermal resistivity, i.e., when \eqref{problem1:4} in the $\e$-problem is replaced with the condition $\theta_\e^f=\theta_\e^s$.
We expect the following changes:}
\begin{itemize}
    \item Case (a): There are two relevant changes. \Cref{thm_d:hom4} reduces to $\theta^s=\theta^f$ on $S\times\Omega^p\times Y^s$ and \cref{thm_d:hom1b} instead reads as
    \[
    |Y^f|\rho^fc^f\partial_t\theta^f-\dive(\kappa_h^f\nabla\theta^f-\rho^fc^f\bar{v}_D\theta^f)=|Y^f|f^f+\int_{\Gamma}\kappa_s\nabla\theta^s\cdot n_\Gamma\di{y}\quad\text{in}\ \ S\times\Omega^p.
    \]
    In other words, the temperature continuity is enforced at the microscale via a Dirichlet boundary condition on $\Gamma$ and the overall energy balance via a source contribution on the macroscale. 
    This is a non--standard configuration (high contrast inclusions with perfect interface transmission) potentially without any clear relevance for applications and, therefore, without many examples in the literature.
    In \cite{Yeh2009}, a very similar setup is chosen (in a stationary setting) with structurally the same limit coupling (continuity equation plus flux contribution in macroscopic equation). 
    \item Case (b): This type of setup seems more reasonable for perfect transmission and has been studied in the literature for other problems. By introducing mixture properties, e.g., $\widetilde{\rho c}=|Y^f|\rho^fc^f+|Y^s|\rho^sc^s$, the limit in the porous part is a one-temperature model for the temperature $\tilde{\theta}$ instead of \cref{homb_porous}:
    \begin{subequations}
    \begin{equation}\label{eq:alternative1}
    \widetilde{\rho c}\partial_t\tilde{\theta}-\dive(\tilde{\kappa_h}\nabla\tilde{\theta}-\rho^f c^f\bar{v}_D\tilde{\theta})=\tilde{f} \quad\text{in}\ \ S\times\Omega^{p}.
    \end{equation}
    Analogously, \cref{homb_porous2} becomes
    \begin{equation}\label{eq:alternative2}
    -\tilde{\kappa_h}\nabla\tilde{\theta}\cdot n_\Sigma=-\kappa^f\nabla\theta^{ff}\cdot n_\Sigma\quad\text{on}\ \ S\times\Sigma
    \end{equation}
    \end{subequations}
    and \cref{homb_porous3} reduces to $\tilde{\theta}=\theta^f$ on $S\times\Sigma$.
    \menew{This is very close to the system established in \cite{OT98b} via RVE averaging.
    The main difference between the limit model presented in \cite{OT98b} and our homogenization limit (where \cref{homb_porous,homb_porous2} are replaced by \cref{eq:alternative1,eq:alternative2}) is an additional flux jump term which is supposed to account for some deviations from thermal equilibrium in the boundary region.}
\end{itemize}
\end{remark}


\section{Homogenization}\label{sec:homogenization}
In this section, we \tf{prove in detail} the homogenization of system \eqref{problem1} in the case of disconnected inclusions (Case (a), see \cref{sec:disc}) and connected solid matrix (Case (b), see \cref{sec:con}).
This is done in the context of two-scale convergence, see, e.g.,~\cite{A92,LNW02} for an overview.

With both geometries, we have to deal with perforated domains depending on $\e$.
Generally speaking, extending the functions and their gradients trivially by zero is sufficient for linear problems like ours: for every function $\phi$ defined on either $\Omega_\e^f$ or $\Omega_\e^s$, $\widehat{\phi}$ denotes the zero extension to the whole of $\Omega$ or $\Omega^p$, respectively.
However, \me{to establish} the continuity conditions for the fluid temperatures at the interface $\Sigma$, we also \me{require} uniform $H^1$-extension operators.

\begin{lemma}[Extension operators]\label{ext_op}
    There is a family of linear extension operators ${\mathcal{P}_\e\colon H^1(\Omega_\e^f)\to H^1(\Omega)}$ such that
    \[
    \|\mathcal{P}_\e\phi\|_{H^1(\Omega)}\leq C_{ext}\|\phi\|_{H^1(\Omega_\e^f)}\quad (\phi\in H^1(\Omega_\e^f))
    \]
    where $C_{ext}>0$ does not depend on $\e$.
\end{lemma}
\begin{proof}
    In the case of disconnected inclusions, these operators are readily available, see, e.g., \cite[Section 2.3]{CS99}.

    In the second case, the situation is a bit more complicated as it is not immediately clear how to extend a function $\phi\in H^1(\Omega_\e^f)$ to the whole of $\Omega$ since \menew{$\Sigma_\e^s=\Sigma\cap\overline{\Omega_\e^s}\neq\emptyset$}.
    In our specific situation \menew{where our porous domain $\Omega^p$ is a finite union of axis-parallel cubes with corner coordinates in $Z^d$}, however, this can still be handled (albeit with more technical and involved proofs). 
    For a concrete reference, we point to~\cite[Theorem 2.2]{HB14}.
\end{proof}

\me{
\begin{remark}\label{rem:strategy}
    The general strategy for the homogenization procedure via the two-scale convergence method is almost always the same and consists of the following steps (cf. \cite{AMP10,EM17,EH15,GP14}):
    \begin{enumerate}
        \item Deduce the existence of limits for the $\e$-dependent functions and their gradients (\cref{lem:limits} in Case (a) and \cref{d:lem:limits} in Case (b)).
        \item \menew{Conduct a limit procedure} with the individual terms of the weak formulations using a specific class of test functions in the same way as in the definition of two-scale convergence, see \cref{twoscale}.
        \menew{This procedure starts} with \cref{eq:limit_heat-1} in Case (a) and \cref{con:eq:limit_heat-5} in Case (b).
        \item \menew{Decouple} the general limit into individual problems and cell problems by careful consideration of the involved test functions plus density arguments.
        \menew{This leads} to system \cref{system_casea} in Case (a) and \cref{system:caseb} in Case (b).
    \end{enumerate}
    In our case of a system of linear parabolic problems, most of the actual limiting process (\menew{Steps $1$ and $2$}) is standard\tfnew{,} although some additional care is needed to ensure that the test functions are continuous across $\Sigma$.
    \end{remark}
}

\subsection{Case (a): Disconnected solid inclusions}\label{sec:disc}
Owing to \cref{ex_est}, we have unique solutions $(\theta_\e^f,\theta_\e^s)\in W_\e$ which satisfy the $\e$-uniform estimate
\begin{equation*}
\|\theta_\e^f\|^2_{L^\infty(S;L^2(\Omega_\e^f))}+\|\theta_\e^s\|^2_{L^\infty(S;L^2(\Omega_\e^s))}
+\|\nabla\theta_\e^f\|^2_{L^2(S\times \Omega_\e^f)}+\e^2\|\nabla\theta_\e^s\|^2_{L^2(S\times \Omega_\e^s)}
+\e\|\theta_\e^f-\theta_\e^s\|^2_{L^2(S\times\Gamma_\e)}\leq C.
\end{equation*}  
Based on these estimates, we are able to \menew{deduce the existence} of limit functions for $\e\to0$.
\menew{This relies on a typical compactness argument for two-scale convergence -- any $L^2$ bounded sequence has a subsequence which converges in the two-scale sense (see \cref{twoscale}).
For the compactness principle as well as further details, we refer to \cite{A92,LNW02}.}

\begin{lemma}\label{lem:limits}
There are limit functions $\theta^{f}\in L^2(S;H^{1}(\Omega^p))$, $\theta^{ff}\in L^2(S;H^{1}(\Omega^{ff}))$, $\theta^{s}\in L^2(S\times\Omega^p; H^1_\#(Y))$, as well as $\theta_1^{f}\in L^2(S\times\Omega^p; H^1_\#(Y))$ and $\theta_1^{ff}\in L^2(S\times\Omega^{ff}; H^1_\#(Y))$ such that
\begin{alignat*}{2}
i&)\ {\widehat{\theta}_\e^f}{}_{|\Omega^p}\overset{2}{\rightharpoonup} \chi^{f}\theta^f ,\qquad &ii&)\  {\nabla\widehat{\theta}_\e^f}_{|\Omega^p}\overset{2}{\rightharpoonup}\chi^{f}\left(\nabla \theta^f+\nabla_y\theta_1^{f}\right),\\
iii&)\ {\theta_\e^{f}}_{|\Omega^{ff}}\overset{2}{\rightharpoonup}\theta^{ff} ,\qquad &iv&)\  \nabla{\theta_\e^{f}}_{|\Omega^{ff}}\overset{2}{\rightharpoonup}\nabla \theta^{ff}+\nabla_y\theta_1^{ff},\\
v&)\  \widehat{\theta}_\e^s\overset{2}{\rightharpoonup}\chi^{s}\theta^s ,\qquad &vi&)\  \e\nabla  \widehat{\theta}_\e^s\overset{2}{\rightharpoonup}\chi^{s}\nabla_y\theta^s
\end{alignat*}
at least up to a subsequence.
\menew{Moreover, the fluid temperature is continuous across the interface $\Sigma$}, that is, $\theta^f=\theta^{ff}$ on $S\times\Sigma$.
In addition, for the interface integral over $\Gamma_\e$, we have
\begin{equation}\label{d:surface}
\e\int_{\Gamma_\e}(\theta_\e^f-\theta_\e^s)\varphi_\e(x,\nicefrac{x}{\e})\di{\sigma}
\to\int_{\Omega^p}\int_\Gamma(\theta^f-\theta^s)\varphi(x,y)\di{\sigma}\di{x}
\end{equation}
for all admissible test functions $\varphi$.
\end{lemma}

\begin{proof}
The two-scale limits $(i)$--$(vi)$ follow directly from the $\e$-uniform estimates given via \cref{ex_est}, see, e.g., \cite[\menew{Theorem 1.2 and Proposition 1.14}]{A92}.

With the use of the extension operators given via \cref{ext_op} and the corresponding estimate for $\|\mathcal{P}_\e\theta_\e^f\|_{H^1(\Omega)}$ based on the a priori estimates for $\theta_\e^f$, we can conclude the existence of $\vartheta\in H^1(\Omega)$ such that   $\mathcal{P}_\e\theta_\e^f\to \vartheta$ converges weakly in $H^1(\Omega)$ along a subsequence.
Due to the compact embedding $H^1(\Omega)\hookrightarrow L^2(\Omega)$, this implies strong convergence $\mathcal{P}\theta_\e^f\to\vartheta$ in $L^2(\Omega)$.
As a consequence, we \me{can infer} that $\vartheta_{|\Omega^{ff}}=\theta^{ff}$ as well as $\vartheta_{|\Omega^{p}}=\theta^f$ due to $\chi_\e^f\twosc\chi^f$.
The temperatures are therefore continuous across the interface $\Sigma$.
For the surface integral limit \eqref{d:surface}, we refer to \cite[Theorem 2.39 (iii)]{PS08}.
\end{proof}

With these limits in mind, we are now passing to the limit $\e\to0$.
To that end, let $\varphi_0^{f}\in C^\infty(\overline{S\times\Omega})$ and $\varphi_1^r\in C^\infty(\overline{S\times\Omega};C_\#(Y^r))$ ($r=s,f$) satisfying $\varphi_0^{f}(T)=0$ and $\varphi_1^r(T)=0$.
In addition, let $\varphi_1^{ff}\in C^\infty(\overline{S\times\Omega};C_\#(Y))$ such that $\varphi_1^f(\cdot,x,y)=\varphi_1^{ff}(\cdot,x,y)$ for all $(x,y)\in\Sigma\times\Sigma_0$ as well as $\varphi_1^{ff}(T)=0$.
We take as test functions $\varphi_\e^r\colon S\times\Omega_\e^r\to\R$ ($r=s,f$) defined via 
\begin{align*}
\varphi_\e^{f}(t,x)&=\varphi^{f}_0(t,x)+\e\begin{cases}\varphi_1^{f}(t,x,\nicefrac{x}{\e})& (t,x)\in S\times \menew{(}\Omega_\e^f\cap\Omega^p\menew{)},\\
\varphi_1^{ff}(t,x,\nicefrac{x}{\e})& (t,x)\in S\times \Omega^{ff}
\end{cases},\\
\varphi_\e^{s}(t,x)&=\varphi_1^s(t,x,\nicefrac{x}{\e}).
\end{align*}
Due to the equality of $\varphi_1^f$ and $\varphi_1^{ff}$ \menew{for all $(x,y)\in\Sigma\times\Sigma_0$}, $\varphi_\e^f$ is continuous across $\Sigma$ thereby satisfying $\varphi_\e^f(t,\cdot)\in H^1(\Omega_\e^f)$.
With the two-scale limits of $\theta_\e^f$ and $\theta_\e^s$ (as established in \cref{lem:limits}), the limits $\e\to0$ in the weak formulation \cref{eps_weak_form} can be evaluated.
For the time derivatives, we find that
\begin{subequations}
\begin{align}
\int_{\Omega}\chi_\e^f\rho^fc^f\theta_\e^f\partial_t\varphi_\e^f\di{x}
    &\to\int_{\Omega^p}|Y^f|\rho^fc^f\theta^f\partial_t\varphi_0^f\di{x}
    +\int_{\Omega^{ff}}\rho^fc^f\theta^{ff}\partial_t\varphi_0^f\di{x},\label{eq:limit_heat-1}\\
\int_{\Omega}\chi_\e^s\rho^sc^s\theta_\e^s\partial_t\varphi_\e^s\di{x}
    &\to\int_{\Omega^p\times Y^s}\rho^sc^s\theta^s\partial_t\varphi_1^s\di{(y,x)}\label{eq:limit_heat-2}.
\end{align}
Here, in \eqref{eq:limit_heat-1}, we have used that both $\theta^f$ and $\varphi_0^f$ are independent of $y\in Y$.
For the diffusive flux terms, we similarly get
\begin{align}
\int_{\Omega}\chi_\e^f\kappa^f\nabla\theta^f_\e\cdot\nabla\varphi_\e^f\di{x}
    &\to\int_{\Omega^p\times Y^f}\kappa^f(\nabla\theta^f+\nabla_y\theta_1^f)\cdot(\nabla\varphi_0^f+\nabla_y\varphi_1^f)\di{(x,y)}\nonumber\\
    &\hspace{1cm}+\int_{\Omega^p\times Y}\kappa^f(\nabla\theta^{ff}+\nabla_y\theta_1^{ff})\cdot(\nabla\varphi_0^f+\nabla_y\varphi_1^{ff})\di{(x,y)},\label{eq:limit_heat-3}\\
\e^2\int_{\Omega}\chi_\e^s\kappa^s\nabla\theta^s_\e\cdot\nabla\varphi_\e^s\di{x}
    &\to\int_{\Omega^p\times Y^s}\kappa^s\nabla_y\theta^s\cdot\nabla_y\varphi_1^s\di{(x,y)}.\label{eq:limit_heat-5}
\end{align}
For the convective flux term, \menew{we make use of our assumption} of strong convergence of ${v_\e}_{|\Omega^{ff}}$ to $v$ and the two-scale convergence of ${v_\e}_{|\Omega^p}$ to $v_D$ (see Assumption (A7)).
Also, $(\mathcal{P}_\e{\theta_\e^f})_{|\Omega^p}$ converges strongly to $\theta^f$ in $L^2(S\times\Omega^p)$ (\me{as shown in the proof of \cref{lem:limits}).
We therefore have a product of a two-scale converging and a strongly converging sequence which converges to the product of the two-scale limit and the strong limit (see \cite[Theorem 1.8]{A92}}):
\begin{equation}
    \begin{split}
        \int_{\Omega}\chi_\e^f\rho^fc^fv_\e\theta^f_\e\cdot\nabla\varphi_\e^f\di{x}
        =\int_{\Omega^p}\rho^fc^fv_\e\mathcal{P}_\e\theta^f_\e\cdot\nabla\varphi_\e^f\di{x}
        +\int_{\Omega^{ff}}\rho^fc^fv_\e\theta^f_\e\cdot\nabla\varphi_\e^f\di{x} \hspace{55pt}\\
    \to\int_{\Omega^p\times Y^f}\rho^fc^fv_D\theta^f\cdot(\nabla\varphi^f+\nabla_y\varphi_1^f)\di{(x,y)}
    +\int_{\Omega^{ff}\times Y}\rho^fc^fv\theta^{ff}\cdot(\nabla\varphi^f+\nabla_y\varphi_1^{ff})\di{(x,y)}.\label{eq:limit_heat-4}
    \end{split}
\end{equation}
In the data terms, namely heat sources and initial conditions, Assumption (A5) and (A6a) allow us to pass to the limit:
\begin{align}
\int_{\Omega}\chi_\e^ff_\e^f\varphi_\e^f\di{x}+\int_{\Omega^p}\chi_\e^ff_\e^s\varphi_\e^s\di{x}
    &\to|Y^f|\int_{\Omega^p}f^f\varphi_0^f\di{x}
    +\int_{\Omega^{ff}}f^{f}\varphi_0^f\di{x}
    +\int_{\Omega^p\times Y^s}f^s\varphi_1^s\di{(x,y)},\label{eq:limit_heat-6}\\
 \int_{\Omega}\chi_\e^f\theta_{\e,0}^f\varphi_\e^f\di{x}+\int_{\Omega^p}\chi_\e^s\theta_{\e,0}^s\varphi_\e^s\di{x}
    &\to|Y^f|\int_{\Omega^p}\theta_0^f\varphi_0^f\di{x}
    +\int_{\Omega^{ff}}\theta_0^{f}\varphi_0^{f}\di{x}
    +\int_{\Omega^p\times Y^s}\theta_0^s\varphi_1^s\di{(x,y)}.\label{eq:limit_heat-7}   
\end{align}
Finally, for the interfacial heat transfer term, we have (\me{cf.~\cref{d:surface}})
\begin{equation}
\e\int_{\Gamma_\e}\alpha(\theta_\e^f-\theta_\e^s)(\varphi_\e^f-\varphi_\e^s)\di{\sigma}
    \to\int_{\Omega^p\times\Gamma}\alpha(\theta^f-\theta^s)(\varphi_0^f-\varphi_1^s)\di{(x,\sigma)}.\label{eq:limit_heat-8}
\end{equation}
\end{subequations}
\noindent
Via a typical density argument \me{(see, e.g., \cite[Theorem 2.3]{A92})}, the limit problem must also hold for all
\begin{align*}
\varphi_0^{f}&\in L^2(S;H^1(\Omega))\quad\text{such that}\ \  \partial_t\varphi_0^{f}\in L^2(S\times\Omega),\\
\varphi_1^{f}&\in L^2(S\times\Omega^p;H^1_\#(Y)),\\
\varphi_1^{ff}&\in L^2(S\times\Omega^{ff};H^1_\#(Y)),\\
\varphi_1^{s}&\in L^2(S\times\Omega^p;H^1_\#(Y))\quad\text{such that}\ \  \partial_t\varphi_1^{s}\in L^2(S\times\Omega^p\times Y),
\end{align*}
with the continuity relation $\varphi_1^f(\cdot,x,y)=\varphi_1^{ff}(\cdot,x,y)$ satisfied for almost all $(x,y)\in\Sigma\times\Sigma_0$.
Again, $\varphi_0^{f}(T)=0$ a.e.~in $\Omega$ and $\varphi_1^{s}(T)=0$ a.e.~in $\Omega^p\times Y$.

\menew{The limits given via \eqref{eq:limit_heat-1}--\eqref{eq:limit_heat-8}, which involve the set of functions $(\theta^f,\theta^f_1,\theta^{ff},\theta_1^{ff},\theta^s)$, constitute the homogenization limit of system \eqref{problem1}.
This system has a rather complex coupling (see, e.g., the diffusive flux \cref{eq:limit_heat-3}) and its interpretation as a model in the context of a heat problem is not obvious.
In particular, it is unclear what the additional functions $\theta^f_1$ and $\theta_1^{ff}$ are actually modeling.
For that reason,}
we want to decouple this limit system with the goal of arriving at a more intuitive description of the effective model.
\menew{Please note, that this decoupling is a standard step in two-scale homogenization, cf.~\cite{EM17,GP14,HJ91}.}

To that end, we start by choosing $\varphi_0^f\equiv0$, $\varphi_1^f\equiv0$, $\varphi_1^{ff}\equiv0$ so that we are left with the solid heat problem
\begin{subequations}
\begin{multline}\label{eq:hom_weak_solid-1}
    -\int_S\int_{\Omega^p\times Y^s}\hspace{-.2cm}\rho^sc^s\theta^s\partial_t\varphi_1^s\di{(x,y)}\di{t}
    +\int_S\int_{\Omega^p\times Y^s}\hspace{-.2cm}\kappa^s\nabla_y\theta^s\cdot\nabla_y\varphi_1^s\di{(x,y)}\di{t}
    +\int_S\int_{\Omega^p\times\Gamma^s}\hspace{-.2cm}\alpha(\theta^s-\theta^f)\varphi_1^s\di{(x,\sigma)}\di{t}\\
    =\int_S\int_{\Omega^p\times Y^s}f^s\varphi_1^s\di{(x,y)}\di{t}
    +\int_{\Omega^p\times Y^s}\rho^sc^s\theta_0^s\varphi_1^s(0)\di{(x,y)}\di{t}
\end{multline}
where the macroscopic variable $x\in\Omega^p$ only acts as a parameter (all derivatives are either with respect to time or the microscopic variable $y\in Y^s$).
We choose a test function $\varphi_0^f$ with compact support in $\Omega^{ff}$ and let $\varphi_1^f\equiv0$, $\varphi_1^{ff}\equiv0$, $\varphi_1^s\equiv0$:   
\begin{equation}\label{eq:hom_weak_solid-2}
    \begin{split}
        -\int_S&\int_{\Omega^{ff}}\rho^fc^f\theta^{ff}\partial_t\varphi_0^f\di{x}\di{t}
        +\int_S\int_{\Omega^{ff}\times Y}\kappa^f(\nabla\theta^{ff}+\nabla_y\theta_1^{ff})\cdot\nabla\varphi_0^f\di{(x,y)}\di{t}\\
        &+\int_S\int_{\Omega^{ff}\times Y}\rho^fc^f\theta^{ff}v\cdot\nabla\varphi_0^{f}\di{(x,y)}\di{t}
        =\int_S\int_{\Omega^{ff}}f^f\varphi_0^f\di{x}\di{t}
        +\int_{\Omega^{ff}}\rho^fc^f\theta_0^{f}\varphi_0^f(0)\di{x}.
    \end{split}
\end{equation}  
Similarly, with $\varphi_0^f$ having compact support in $\Omega^p$:
\begin{equation}\label{eq:hom_weak_solid-3}
    \begin{split}
         -\int_S\int_{\Omega^{p}}|Y^f|\rho^fc^f\theta^f\partial_t\varphi_0^f\di{x}\di{t}
        +\int_S\int_{\Omega^{p}\times Y^f}\kappa^f(\nabla\theta^f+\nabla_y\theta_1^f)\cdot\nabla\varphi_0^f\di{(x,y)}\di{t} \hspace{75pt}
        \\
        +\int_S\int_{\Omega^{p}}\rho^fc^f\theta^f\bar{v}_D\cdot\nabla\varphi_0^f\di{(x,y)}\di{t}
        +\int_S\int_{\Omega^{p}}\int_\Gamma\alpha(\theta^f-\theta^s,\varphi_0^f)\di{\sigma}\di{x}\di{t} \hspace{50pt}
        \\
        =\int_S\int_{\Omega^{p}}|Y^f|f^f\varphi_0^f\di{x}\di{t}
        +\int_{\Omega^{p}}|Y^f|\rho^fc^f\theta_0^{f}\varphi_0^f(0)\di{x},   
    \end{split}
\end{equation}
where we have set
\[
\bar{v}_D(t,x)=\int_{Y^f}v_D(t,x,t)\di{y}.
\]
Next, choosing $\varphi_0^f\equiv0$ and $\varphi_1^s\equiv0$, we get the elliptic problem
\begin{multline*}
    \kappa^f(\nabla\theta^f+\nabla_y\theta_1^f,\nabla_y\varphi_1^f)_{\Omega^p\times Y^f}
    +\kappa^f(\nabla\theta^{ff}+\nabla_y\theta_1^{ff},\nabla_y\varphi_1^{ff})_{\Omega^{ff}\times Y}\\
    +(\rho^fc^fv_D\theta^f,\nabla_y\varphi_1^f)_{\Omega^p\times Y^f}
    +(\rho^fc^fv\theta^{ff},\nabla_y\varphi_1^{ff})_{\Omega^{ff}\times Y^f}
    =0,
\end{multline*}
where we are allowed to vary test functions $\varphi_1^f$ and $\varphi_1^{ff}$ freely as long as the compatibility condition $\varphi_1^f=\varphi_1^{ff}$ is satisfied almost everywhere on $\Sigma\times\Sigma_0$.
As a consequence, we can decouple this elliptic problem into two separate problems (note that \me{both} $x\in\Omega^p$ \me{as well as} $x\in\Omega^{ff}$ implies $x\notin\Sigma$)
\begin{align}
     \kappa^f(\nabla\theta^f+\nabla_y\theta_1^f,\nabla_y\varphi_1^f)_{S\times\Omega^p\times Y^f}
    +(\rho^fc^fv_D\theta^f,\nabla_y\varphi_1^f)_{S\times\Omega^p\times Y^f}
    &=0,\label{eq:hom_weak_solid-4}\\
    \kappa^f(\nabla\theta^{ff}+\nabla_y\theta_1^{ff},\nabla_y\varphi_1^{ff})_{S\times\Omega^{ff}\times Y}
    +(\rho^fc^fv\theta^{ff},\nabla_y\varphi_1^{ff})_{S\times\Omega^{ff}\times Y}
    &=0.  \label{eq:hom_weak_solid-5} 
\end{align}
Since $\varphi_1^{ff}$ is $Y$-periodic, we have $\int_Y\nabla_y\varphi_1^{ff}\di{y}=0$, and since $v$ and $\theta^{ff}$ are $y$-independent, \cref{eq:hom_weak_solid-5} simplifies to
\begin{equation*}
\kappa^f(\nabla_y\theta_1^{ff},\nabla_y\varphi_1^{ff})_{Y}
    =0\quad\text{a.e.~in}\ S\times\Omega^p.
\end{equation*}
This elliptic problem has only \tfnew{constant solutions in the space of periodic functions,} implying that $\nabla_y\theta_1^{ff}\equiv0$.
\menew{In the porous region (\cref{eq:hom_weak_solid-4}), we find that
\[
\int_{\Omega^P}\rho^fc^f\theta_f\int_{Y^f}v_D\cdot\nabla_y\varphi_1^f\di{y}\di{x}=0
\]
since $v_D(t,x,y)\cdot n_\Gamma(y)=0$ almost everywhere on $S\times\Omega^p\times\Gamma$  as well as $\dive_y(v_D)=0$ in $S\times\Omega\times Y^f$ (Assumption (A7)).
As time is only a parameter in \cref{eq:hom_weak_solid-4}, we localize in time:
\begin{align}\label{eq:hom_weak_solid-6} 
     (\kappa^f(\nabla\theta^f+\nabla_y\theta_1^f),\nabla_y\varphi_1^f)_{\Omega^p\times Y^f}
    &=0\quad\text{a.e.~in}\ S.
\end{align}   
}
\end{subequations} 

The homogenization limit therefore can equivalently be formalized in the following weak system (summarizing \cref{eq:hom_weak_solid-1,eq:hom_weak_solid-2,eq:hom_weak_solid-3,eq:hom_weak_solid-5,eq:hom_weak_solid-6}):\footnote{Variational equalities \eqref{d:hom_var3} and \eqref{d:hom_var4} hold almost everywhere in $\Omega^p$ and $S$, respectively.}
\begin{subequations}
\begin{equation}\label{d:hom_var1}
    \begin{split}
       \int_S\int_{\Omega^{ff}}-\rho^fc^f\theta^{ff}\partial_t\varphi_0^f+\left(\kappa^f\nabla\theta^{ff}-\rho^fc^fv\theta^{ff}\right)\cdot\nabla\varphi_0^{f}\di{x}\di{t} \hspace{65pt}
       \\
        =\int_S\int_{\Omega^{ff}}f^f\varphi_0^f\di{x}\di{t}+\int_{\Omega^{ff}}\rho^fc^f\theta_0^{ff}\varphi_0^f(0)\di{x}, 
    \end{split}
\end{equation}
\begin{equation}\label{d:hom_var2}
    \begin{split}
        -\int_S\int_{\Omega^p}|Y^f|\rho^fc^f\theta^f\partial_t\varphi_0^f\di{x}\di{t}
        +\int_S\int_{\Omega^p\times Y^f}\kappa^f(\nabla\theta^f+\nabla_y\theta_1^f)\cdot\nabla\varphi_0^f\di{(x,y)}\di{t} \hspace{55pt}
        \\
        +\int_S\int_{\Omega^p}\rho^fc^f\bar{v}_D\theta^f\cdot\nabla\varphi_0^f\di{x}\di{t}
        +\int_S\int_{\Omega^p\times\Gamma}\alpha(\theta^f-\theta^s)\varphi_0^f\di{(x,\sigma)}\di{t} \hspace{35pt}
        \\
        =\int_S\int_{\Omega^p}|Y^f|f^f\varphi_0^f\di{x}\di{t}+\int_{\Omega^p}|Y^f|\rho^fc^f\theta_0^f\varphi_0^f(0)\di{x},
    \end{split}
\end{equation}
\begin{equation}\label{d:hom_var3}
    \begin{split}
        &-\int_S\int_{\Omega^p\times Y^s}\rho^sc^s\theta^s\partial_t\varphi_1^s\di{(x,y)}\di{t}
        +\int_S\int_{\Omega^p\times Y^s}\kappa^s\nabla_y\theta^s\cdot\nabla_y\varphi_1^s\di{(x,y)}\di{t} 
        \\
        &+\int_S\int_{\Omega^p\times\Gamma}\alpha(\theta^s-\theta^f)\varphi_1^s\di{(x,\sigma)}\di{t}
        =\int_S\int_{\Omega^p\times Y^s}f^s\varphi_1^s\di{(x,y)}\di{t}
        +\int_{\Omega^p\times Y^s}\rho^sc^s\theta_0^s\varphi_1^s(0)\di{(x,y)},
    \end{split}
\end{equation}
\begin{equation}\label{d:hom_var4}
    \int_{\Omega^p\times Y^f}\kappa^f(\nabla\theta^f+\nabla_y\theta_1^f)\cdot\nabla_y\varphi_1^f\di{(x,y)}\di{t}
    =0\quad\text{a.e.~in}\ S
\end{equation}
\end{subequations}
for all appropriate test functions.
We want to further decouple this problem by eliminating $\theta_1^f$ from the system.
Introducing cell solutions $\xi_i\in H_\#^1(Y^f)$, $i=1,2,3$, as the unique, zero-average solution of
\begin{alignat}{2}\label{d:cell1}
\int_{Y^f}(\nabla_y\xi_i^f+e_i)\cdot\nabla_y\phi\di{y}=0\qquad (\phi\in H_\#^1(Y^f)).
\end{alignat}
\tfnew{S}etting ($\xi^f=(\xi_1^f,\xi_2^f,\xi_3^f)$)
\[
\tau(t,x,y):=\xi^f(y)\cdot\nabla\theta^f(t,x)-\theta_1^f(t,x,y),
\]
we are able to calculate
\begin{multline*}
(\kappa^f(\nabla\theta^f+\nabla_y\theta_1^f),\nabla_y\varphi_1^f)_{\Omega^p\times Y^f}
  =\sum_{i=1}^3\int_{\Omega^p}\kappa^f\partial_i\theta^f\int_{Y^f}(\nabla_y\xi_i^f+e_i)\cdot\nabla_y\varphi_1^f\di{y}\di{x}\\
  -\int_{\Omega_p\times Y^f}\nabla_y\tau\cdot\nabla_y\varphi_1^f\di{(x,y)},
\end{multline*}
where the first term on the right hand side vanishes as $\xi_i^f$ \tfnew{solves} Problem \eqref{d:cell1}.
The function $\tau$ thus satisfies
\[
\int_{\Omega_p\times Y^f}\nabla_y\tau\cdot\nabla_y\varphi_1^f\di{(x,y)}
\]
which implies that $\tau$ is constant in $y\in Y^f$.
As a result, we can characterize $\theta_1^f(t,x,y)$ via the relation
\[
\theta_1^f(t,x,y)=\xi^f(y)\cdot\nabla\theta^f(t,x)+r(t,x)
\]
for some function $r$.
For the diffusive flux term in \eqref{d:hom_var2}, we then get
\[
\int_{\Omega^p\times Y^f}\kappa^f(\nabla\theta^f+\nabla_y\theta_1^f)\cdot\nabla\varphi_0^f\di{(x,y)}\\
=\int_{\Omega^p}\kappa^f\sum_{i=1}^3\partial_i\theta^f\int_{Y^f}(e_i+\nabla_y\xi_i)\di{y}\cdot\nabla\varphi_0^f\di{x}.
\]
Introducing the \menew{standard effective diffusivity} $\kappa^h\in\R^{3\times3}$ \menew{(cf.~\cite[Definition 1.2]{A92} or \cite[Section 2.6]{HJ91})} via
\begin{align*}
\kappa^h_{ij}&=\kappa^f\int_{Y^f}(\nabla_y\xi_i^f+e_i)\cdot e_j\di{y},
\end{align*}
the diffusive flux simplifies to
\begin{equation}\label{d:hom_diff}
\int_{\Omega^p\times Y^f}\kappa^f(\nabla\theta^f+\nabla_y\theta_1^f)\cdot\nabla\varphi_0^f\di{(x,y)}
=\int_{\Omega^p}\kappa^h\nabla\theta^f\cdot\nabla\varphi_0^f\di{x}.
\end{equation}
Summarizing these results, we are finally led to the following system of partial differential equations 
\begin{subequations}\label{system_casea}
\begin{alignat}{2}
\rho^fc^f\partial_t\theta^{ff}-\dive(\kappa^{f}\nabla\theta^{ff}-\rho^fc^fv\theta^{ff})&=f^{f}&\quad&\text{in}\ \ S\times\Omega^{ff},\label{a:hom1}\\
|Y_f|\rho^fc^f\partial_t\theta^{f}-\dive(\kappa^{h}\nabla\theta^{f}-\rho^fc^f\bar{v}_D\theta^{f})+\alpha\int_\Gamma(\theta^f-\theta^s)&=|Y_f|f^f&\quad&\text{in}\ \ S\times\Omega^{p}, \label{b:hom1}\\
\rho^s c^s \partial_t\theta^s-\kappa^s\Delta_y\theta^s&=f^s&\quad&\text{in}\ \ S\times\Omega^{p}\times Y^s. \label{c:hom1}
\end{alignat}
\menew{These are} supplemented by conditions at the interfaces $\Sigma$ and $\Gamma$ (note that $(\overline{v}_D-v)\cdot n_\Sigma=0$ on $S\times\Sigma$ due to Assumption (A7))
\begin{alignat}{2}
\theta^f&=\theta^{ff}&\quad&\text{on}\ \ S\times\Sigma,\label{d:hom1}\\
-\kappa^{h}\nabla\theta^{f}\cdot n_\Sigma&=-\kappa^f\nabla\theta^{ff}\cdot n_\Sigma&\quad&\text{on}\ \ S\times\Sigma, \label{e:hom1}\\
-\kappa^s\nabla\theta^s\cdot n_\Gamma&=\alpha(\theta^s-\theta^f)&\quad&\text{on}\ \ S\times\Omega^{p}\times\Gamma,
\end{alignat}
as well as initial and \menew{boundary conditions posed on the external boundaries (i.e., $\partial\Omega^{ff}\setminus\Sigma$ and $\partial\Omega^{p}\setminus\Sigma$)}
\begin{alignat}{2}
\theta^{ff}&=\theta_0^{f}&\quad&\text{in}\ \ \Omega^{ff},\\
\theta^f&=\theta_0^{f}&\quad&\text{in}\ \ \Omega^p,\\
\theta^s&=\theta_0^{s}&\quad&\text{in}\ \ \Omega^p\times Y^s,\\
-\kappa^{f}\nabla\theta^{ff}\cdot\nu&=0&\quad&\text{on}\ \ S\times(\partial\Omega^{ff}\setminus\Sigma),\\
-\kappa^{h}\nabla\theta^{f}\cdot \nu&=0&\quad&\text{on}\ \ S\times(\partial\Omega^{p}\setminus\Sigma).
\end{alignat}
\end{subequations}


\subsection{Case (b): Connected solid matrix.}\label{sec:con}
We \me{adopt} a similar approach \menew{as in} the preceding section, with the main distinction being the connectedness of the solid matrix in the porous medium.
As a result, we begin with a slightly different $\e$ scaling and need to be more careful with the interface limits.
While the final homogenized model is structurally different in this scenario, the arguments for the limit $\e\to0$ can \tfnew{mostly} be transferred directly from the previous section.
Again, we \menew{start with the existence of} two-scale limits:
    
\begin{lemma}\label{d:lem:limits}
There are limit functions $\theta^{f}\in L^2(S;H^{1}(\Omega^p))$, $\theta^{ff}\in L^2(S;H^{1}(\Omega^{ff}))$, $\theta^{s}\in L^2(S;H^1(\Omega^p))$, as well as $\theta_1^{f}\in L^2(S\times\Omega^p; H^1_\#(Y))$, $\theta_1^{s}\in L^2(S\times\Omega^p; H^1_\#(Y))$ and $\theta_1^{ff}\in L^2(S\times\Omega^{ff}; H^1_\#(Y))$ such that
\begin{alignat*}{2}
i&)\ \widehat{\theta}_\e^f\overset{2}{\rightharpoonup} \chi^{f}\theta^f ,\qquad &ii&)\  \nabla\widehat{\theta}_\e^f\overset{2}{\rightharpoonup}\chi^{f}\left(\nabla \theta^f+\nabla_y\theta_1^{f}\right),\\
iii&)\ \theta_\e^{ff}\overset{2}{\rightharpoonup}\theta^{ff} ,\qquad &iv&)\  \nabla\theta_\e^{ff}\overset{2}{\rightharpoonup}\nabla \theta^{ff}+\nabla_y\theta_1^{ff},\\
v&)\  \widehat{\theta}_\e^s\overset{2}{\rightharpoonup}\chi^{s}\theta^s ,\qquad &vi&)\  \nabla  \widehat{\theta}_\e^s\overset{2}{\rightharpoonup}\chi^{s}\left(\nabla\theta^s+\nabla_y\theta_1^s\right).
\end{alignat*}
At the interface $\Sigma$, we have continuity of the fluid temperature, that is, $\theta^f=\theta^{ff}$ on $S\times\Sigma$.
In addition, for the interior part of the interface integral ($\Gamma_\e^{int}=\Gamma_\e\cap\Omega^p$), we have
\[
\e\int_{\Gamma_\e^{int}}(\theta_\e^f-\theta_\e^s)\varphi_\e(x,\nicefrac{x}{\e})\di{\sigma}
\to\int_{\Omega^p}(\theta^f-\theta^s)\int_\Gamma\varphi(x,y)\di{\sigma}\di{x}
\]
and for the exterior part, $\Sigma_\e^s$, it holds
\[
\int_{\Sigma_\e^s}(\theta_\e^f-\theta_\e^s)\varphi_\e(x,\nicefrac{x}{\e})\di{\sigma}
\to\int_{\Omega^p}(\theta^f-\theta^s)\int_{\Sigma^s}\varphi(x,y)\di{\sigma}\di{x}
\]
for all admissible test functions $\varphi$.
\end{lemma}

\begin{proof}
The proof of this lemma is similar to the one of \cref{lem:limits}: the limits follow directly from the $\e$-uniform estimates given via \cref{ex_est}, see, e.g., \cite[\menew{Theorem 1.2 and Proposition 1.14}]{A92}, and the continuity via \tfnew{the extension operator presented in} \cref{ext_op}.

For the exterior part, we note that in our specific geometric setup\footnote{Flat interface $|\Sigma|$ where the $\e$ are chosen in a way to perfectly tile the domain with $\e Y$-cells.} the characteristic function of $\Sigma_\e^s$ converges to $|\Sigma^s|$ weakly in $L^2(\Sigma)$, \cite[Theorem 3]{GP14}.
The limit then follows since $\theta^f,\theta^s$ are $y$-independent.
\end{proof}    

Let $\varphi_0^{r}\in C^\infty(\overline{S\times\Omega})$ and $\varphi_1^r\in C^\infty(\overline{S\times\Omega};C_\#(Y^r))$ ($r=s,f$).
Also, let $\varphi_1^{ff}\in C^\infty(\overline{S\times\Omega};C_\#(Y))$ such that $\varphi_1^f(\cdot,x,y)=\varphi_1^{ff}(\cdot,x,y)$ for all $(x,y)\in\Sigma\times\Sigma_0$.
We also \me{assume} $\varphi_0^{r}(T)=0$, $\varphi_1^r(T)=0$, and $\varphi_1^{ff}(T)=0$.
We take as test functions $\varphi_\e^r\colon S\times\Omega_\e^r\to\R$ defined via 
\begin{align*}
\varphi_\e^{f}(t,x)&=\varphi^{f}_0(t,x)+\e\begin{cases}\varphi_1^{f}(t,x,\nicefrac{x}{\e})& (t,x)\in S\times \Omega_\e^f\cap\Omega^p,\\
\varphi_1^{ff}(t,x,\nicefrac{x}{\e})& (t,x)\in S\times \Omega^{ff}
\end{cases},\\
\varphi_\e^{s}(t,x)&=\varphi_0^s(t,x)+\e\varphi_1^s(t,x,y).
\end{align*}
Due to the equality of $\varphi_1^f$ and $\varphi_1^{ff}$ on $(x,y)\in\Sigma\times\Sigma_0$, we find that $\varphi_\e^f$ is continuous across $\Sigma_\e^f$ thereby satisfying $\varphi_\e^f(t,\cdot)\in H^1(\Omega_\e^f)$.
The limits mostly follow with the same arguments as for their counterparts in the previous section.
For the diffusive flux in the solid medium, we get
\begin{align}
\int_{\Omega_\e^s}\kappa^s\nabla\theta^s_\e\cdot\nabla\varphi_\e^s\di{x}
    &\to\int_{\Omega^p\times Y^s}\kappa^s(\nabla\theta^s+\nabla_y\theta_1^s)\cdot(\nabla\varphi_0^s+\nabla_y\varphi_1^s)\di{(x,y)}.\label{con:eq:limit_heat-5}
\end{align}
Introducing additional cell solutions $\xi^s_i\in H_\#^1(Y^f)$, $i=1,2,3$, as the unique, zero-average solution\tfnew{s} of
\begin{alignat}{2}\label{c:cell1}
\int_{Y^s}(\nabla_y\xi^s_i-e_i)\cdot\nabla_y\phi\di{y}=0\qquad (\phi\in H_\#^1(Y^s)),
\end{alignat}
we again can argue that
\[
\theta_1^s(t,x,y)=\xi^s(y)\cdot\nabla\theta^s(t,x)+r(t,x),
\]
where the function $r$ does not depend to $y\in Y^s$.
With this we can introduce the homogenized heat conductivities $\kappa_h^f,\kappa_h^s\in\R^{3\times3}$ with entries
\begin{equation*}
(\kappa_h^f)_{ij}=\kappa^f\int_{Y^f}(\nabla_y\xi^f_i+e_i)\cdot e_j\di{y},\qquad
(\kappa_h^s)_{ij}=\kappa^s\int_{Y^s}(-\nabla_y\xi^s_i+e_i)\cdot e_j\di{y}.
\end{equation*}
Please note that the definition of $\kappa_h^f$ is identical with its counterpart from \cref{sec:disc} (of course, the value most certainly will be different due to the changes in geometry).
Focusing on the interfacial heat transfer term, we have
\begin{equation*}
\alpha_\e(\theta_\e^f-\theta_\e^s,\varphi_\e^f-\varphi_\e^s)_{\Gamma_\e}
=\e\alpha(\theta_\e^f-\theta_\e^s,\varphi_\e^f-\varphi_\e^s)_{\Omega^p\cap\Gamma_\e}
+\alpha(\theta_\e^f-\theta_\e^s,\varphi_\e^f-\varphi_\e^s)_{\Sigma_\e^s}.
\end{equation*}
The individual parts converge
\begin{equation}
\e\alpha(\theta_\e^f-\theta_\e^s,\varphi_\e^f-\varphi_\e^s)_{\Omega^p\cap\Gamma_\e}\to\alpha|\Gamma|(\theta^f-\theta^s,\varphi_0^f-\varphi_0^s)_{\Omega^p},\label{con:eq:limit_heat-8}
\end{equation}
\begin{equation}
\alpha(\theta_\e^f-\theta_\e^s,\varphi_\e^f-\varphi_\e^s)_{\Sigma_\e^s}\to\alpha|\Sigma^s|(\theta^{ff}-\theta^s,\varphi_0^f-\varphi_0^s)_{\Omega^p},\label{con:eq:limit_heat-9}
\end{equation}
With that, we can state the effective system of partial differential equations
\begin{subequations}\label{system:caseb}
\begin{alignat}{2}
\rho^fc^f\partial_t\theta^{ff}-\dive(\kappa^{f}\nabla\theta^{ff}-\rho^fc^fv\theta^{ff})&=f^{ff}&\quad&\text{in}\ \ S\times\Omega^{ff},\\
|Y^f|\rho^fc^f\partial_t\theta^{f}-\dive(\kappa_{h}^f\nabla\theta^{f}-\rho^fc^f\overline{v}_D\theta^{f})+\alpha|\Gamma|(\theta^f-\theta^s)&=|Y^f|f^f&\quad&\text{in}\ \ S\times\Omega^{p},\\
|Y^s|\rho^sc^s\partial_t\theta^s-\dive(\kappa_h^s\nabla\theta^s)-\alpha|\Gamma|(\theta^f-\theta^s)&=|Y^s|f^s&\quad&\text{in}\ \ S\times\Omega^{p},
\end{alignat}
coupled with interface conditions at $\Sigma$
\begin{alignat}{2}
\theta^f&=\theta^{ff}&\quad&\text{on}\ \ S\times\Sigma,\\
-(\kappa_h^s\nabla\theta^s+\kappa_{h}^f\nabla\theta^{f})\cdot n_\Sigma&=-\kappa^f\nabla\theta^{ff}\cdot n_\Sigma&\quad&\text{on}\ \ S\times\Sigma,\\
-\kappa_h^s\nabla\theta^s\cdot n_\Sigma&=\alpha|\Sigma^s|(\theta^s-\theta^f)&\quad&\text{on}\ \ S\times\Sigma,
\end{alignat}
as well as \menew{conditions on the external boundaries (i.e., $\partial\Omega^p\setminus\Sigma$ and $\partial\Omega^{ff}\setminus\Sigma$)} and initial conditions.
Here, we have used $(\overline{v}_D-v)\cdot n_\Sigma=0$ on $S\times\Sigma$ (Assumption (A7)).
\end{subequations}

\section{Simulations}\label{sec:simulations}
In this section, we present numerical simulations \tfnew{to illustrate and verify various aspects of the homogenized models}, including differences, similarities, and other interesting observations.
We begin by analyzing systems without convection \menew{which allows} us to focus on heat diffusion and energy storage within the different domains.
In subsequent experiments, we incorporate convection by employing a combination of Navier-Stokes and Darcy equations to compute the convective effects.

All simulations are carried out with the FEM library FEniCS \cite{Fenics}.
The software Gmsh \cite{gmsh} is utilized to generate meshes of the various pore structures used in the simulations.
The time dependence is handled with the implicit Euler method.
\tf{In the simulations of the homogenized models, temperature fields are represented by piecewise linear polynomials. For fluid flow computations, the Navier-Stokes system and the Darcy equation are employed, with coupling at the interface $\Sigma$ governed by the Beavers-Joseph conditions \cite{BJ67, Vivette2009, JM00}. Consequently, at the interface a jump of the tangential velocity is anticipated. To accurately capture this behaviour, discontinuous Taylor-Hood elements \cite{Sch2004} of both second and first order are employed for the fluid velocity and pressure. The non--linearity of the Navier-Stokes equation is addressed using an Oseen iteration.}
\tfnew{The concrete boundary conditions of the flow are specified in \cref{ch:simu_with_flow}, when the convection will be considered for the first time.}

\tf{
FEniCS, or finite elements in general, use the weak formulation of the problem. For the homogenized models, the weak formulations are presented in the limiting procedure of \cref{sec:homogenization} and are not repeated here.}
\menew{
Within FEniCS, the space discretization for a connected solid matrix, as outlined in Case (b) (\cref{sec:con}), is accomplished by establishing one unified function space for the fluid temperature, merging \(\theta^f\) and \(\theta^{ff}\), and a distinct function space in \(\Omega^p\) for the solid temperature \(\theta^s\). All heat exchange interface conditions for the model naturally manifest in the weak form via the corresponding interface integrals.}
\menew{Constructing the system for the disconnected matrix, as detailed in Case (a) (\cref{sec:disc}), is more complex due to the domain \(Y^s\) for \(\theta^s\) and the interaction between \(\theta^f\) and \(\theta^s\) via the heat exchange at \(\Gamma\).
Subsequently, we utilize a fixed point algorithm at each time step to determine the current temperature.
To do this, we once again create a unified function space for the fluid temperature and select a discrete set of points \(\{\mathbf{x}_i\}_{i=1}^N \subset \Omega^p\).
Rather than solving (\ref{c:hom1}) over the entire domain \(\Omega^p\), we solve only at the specified points \(\mathbf{x}_i\).
}
This leads to \menew{the} $N$ cell problems
\begin{equation}\label{eq:locall_problem}
  \begin{alignedat}{2}
    \partial_t \theta^s_i - \kappa^s \Delta_y \theta_i^s &= f_i^s &\quad&\text{in}\ \ S \times Y^s, \\
    -\kappa_s \nabla_y \theta^s \cdot n_\Gamma &= \alpha(\theta^s_i - \theta_i^f) &\quad&\text{on}\ \ S \times \Gamma, \\
    \theta^s_i(0)&=\theta^s_{0,i}&\quad&\text{in} \ \ Y^s
  \end{alignedat}
\end{equation}
with $f^s_i=f^s(\cdot, \mathbf{x}_i, \cdot)$, $\theta_i^f=\theta^f(\cdot, \mathbf{x}_i)$, and $\theta^s_{0,i}=\theta^s_0(\mathbf{x}_i,\cdot)$.
A straightforward choice for the positions $\{\mathbf{x}_i\}_{i=1}^N$ is the set of mesh vertices, or coordinates of the quadrature formula, belonging to $\Omega^p$.

\tflast{An alternative approach involves selecting only a subset of the mesh vertices and then interpolating the resulting cell solutions.
This reduces the number of local equations and therefore the overall computational demand at the cost of additional \textit{consistency errors}, see also \cref{remark_consistency_error}.
We point to \cite[Section 6.2]{Meier_diss} where this approach is also followed.
In our specific scenarios, we found that relatively small sets of points together with linear interpolation is already sufficient to accurately capture the linear coupling between solid and fluid temperatures.
The specific coordinates and the total number of points for each simulation are dependent on the scenario and mentioned in the ensuing sections.
In more complicated cases, e.g., nonlinear heat exchange or strong convection, linear interpolation might now work.
}

we contrasted cell positions at each vertex against those on grids of different sizes and found that using a limited number of specified cell positions, instead of solving \cref{eq:locall_problem} at each mesh vertex, is sufficient in our simulation studies.
The specific coordinates and the total number of points for each simulation are mentioned in the ensuing sections.

We define ${\Theta^s := \Theta^s[\theta^s_1, \dots, \theta^s_N] : S\times \Omega^p  \times Y^s \to \R}$, such that $\Theta^s(\cdot, \mathbf{x}_i, \cdot) = \theta^s_i$ for all $i=1,\dots,N$, as the representation of the temperatures $\{\theta^s_i\}_{i=1}^N$ on $\Omega^p$.
Additionally, since $\theta^f$ is independent of $y$, the integral in (\ref{b:hom1})
reduces to 
\begin{equation*}
    \int_\Gamma (\theta^f - \Theta^s) = |\Gamma| \theta^f - \int_\Gamma \Theta^s.
\end{equation*}
At each time step we apply \tfnew{Algorithm} \ref{algo:iterative_scheme}, to solve the problem with the disconnected matrix.

\begin{algorithm}[ht]
	\caption{Iterative scheme for the disconnected model}
    \label{algo:iterative_scheme}
	\KwIn{Information $\theta^{ff}_n, \theta^f_n, \{\theta^s_{i,n}\}_{i=1}^N$, at time step $t_n$, and tolerance $\tau$.}
	Do the time step  $t_n \to t_{n+1}$ of \cref{a:hom1,b:hom1}, by replacing $\theta^s$ with $\Theta^s_n[\theta^s_{1,n}, \dots, \theta^s_{N,n}]$, to compute $\theta^{ff}_{0, n+1}, \theta^{f}_{0, n+1}$.\\
    Set $k=0, e_k \geq \tau$. \\
	\While{$e_k \geq \tau$}
    	{\For{$i=1$ \KwTo $N$}{
            Use $\theta^f_{k, n+1}$ to do the time step $t_n \to t_{n+1}$ of the system (\ref{eq:locall_problem}) to compute $\theta_{i, k, n+1}^s$.
        }
        Redo the time step of (\ref{a:hom1})-(\ref{b:hom1}) with $\{\theta_{i, k, n+1}^s\}_{i=1}^N$, denote the solutions by $\theta^{ff}_{k+1, n+1}, \theta^f_{k+1, n+1}$.\\
        Compute $e_{k+1} = \sqrt{\|\theta^{ff}_{k+1, n+1}-\theta^{ff}_{k, n+1}\|_{L^2(\Omega^{ff})}^2 + \|\theta^{f}_{k+1, n+1}-\theta^{f}_{k, n+1}\|_{L^2(\Omega^{p})}^2}$. \\
        Set k = k + 1.}
    \KwRet{$\theta^{ff}_{k, n+1}, \theta^{f}_{k, n+1}$ \text{and} $\{ \theta_{i, k-1,n+1}^s\}_{i=1}^N$.}
\end{algorithm}
\begin{remark}\label{rem:algo_conv}
        \tflast{When \(\{\mathbf{x}_i\}_{i=1}^N \subset \Omega^p\) is equal to the set of mesh vertices,} Algorithm \ref{algo:iterative_scheme} converges if $\Delta t > 0$ is chosen \tfnew{such} that $\tfrac{|Y^f|}{\Delta t} > \rho^f c^f\|\widetilde{v}\|_{L^\infty(S \times \Omega)}^2 \tfrac{1}{2\delta} $. 
        Here, $\widetilde{v} = \chi_{\Omega^{ff}}v + \chi_{\Omega^{p}}\Bar{v}_D$, $0 < \delta < 2 \lambda$  and $0<\lambda<\kappa^f$ \menew{satisfies}
        \begin{equation*}
            \lambda \| \zeta \|^2 \leq \kappa^f_h \zeta\cdot\zeta \quad \text{for all } \zeta \in \R^d.
        \end{equation*}
        Additionally, at iteration step $k\geq 2$, the estimate
        \begin{equation}\label{eq:fixpoint_estimate}
            e_{k} \leq C \left( \frac{\alpha |\Gamma|}{\alpha |\Gamma| + \tfrac{\rho^f c^f |Y^f|}{\Delta t} - (\rho^f c^f)^2 \|\widetilde{v}\|_{L^\infty(S \times \Omega)}^2 \tfrac{1}{2\delta}} \right )^{k-1.5} \|\theta^{f}_{1, n+1}-\theta^{f}_{0, n+1}\|_{L^2(\Omega^{p})},
        \end{equation}
        holds with a constant \tfnew{$C>0$ independent of $k$}. The convergence proof follows with an argument similar to the one of \cref{hom:a} and is shown in \cref{sec:algo_conv}.
        
    To guarantee convergence, the expression inside the brackets in \cref{eq:fixpoint_estimate} has to be smaller than one, which gives a condition for the time step size $\Delta t$ \menew{relative} to the fluid velocity.
    \menew{Also}, larger values of $\alpha$, in case all other parameters are fixed, lead to a slower convergence. \tfnew{This slower convergence, for increasing heat exchange $\alpha$, was also observed in our numerical simulations.
    Lastly, the diffusion term was ignored in the convergence proof.
    If an estimate for the diffusion term in terms of the solution $\theta$, like the Poincar\'e inequality, is available \tflast{the influence of the convection could be reduced and the iterative scheme would also converge for larger time steps $\Delta t$.}}
\end{remark}

For more \menew{intricate} problems, particularly when \tfnew{numerous} local problems (\ref{eq:locall_problem}) have to be solved, \menew{switching} to the memory representation mentioned in \cref{hom:b2} \menew{could reduce} the computational effort.
Moreover, the \tfnew{local} problems can also be solved in parallel since they are independent.

Before we present the simulation results, we \tfnew{first detail the specifics of} our numerical experiments.
We set $\Omega^{ff} = [0, 2] \times [0.5, 1] \times [0, 2]$, $\Omega^p = [0, 2] \times [0, 0.5] \times [0, 2]$ and $S=[0, 20]$.
For the space discretization, we take $H=0.025$, and for the time stepping with the Euler method, $\tf{\Delta t}=0.1$.
The tolerance in our iteration scheme is set at $\tau=10^{-5}$.
Additionally, we assign values $\kappa^f=0.1$, $\kappa^s=0.4$, and $\alpha=0.1$, noting that the value of $\alpha$ will later be varied.
Both densities and heat capacities are normalized to 1.
Various pore structures are considered in our numerical experiments, as depicted in \cref{tab:data_connected,tab:data_non_connected}.
Most geometries are constructed such that either $|Y^f|$ or $|\Gamma|$ are consistent between different cases.
The effective parameters were computed by numerically solving the corresponding cell problems.
To validate the accuracy of these effective parameters, the computations were repeated with a step width of $H=0.0125$.
The relative difference was consistently below 2\,\% for all examples.

\begin{table}[htp]
    \centering
    \footnotesize{\begin{tabular}{c|c|c|c}
         Name tag & (DC1) & (DC2) & (DC3) \\
         \hline
         &&& \\[\dimexpr-\normalbaselineskip+2pt]
         \includegraphics[width=18mm, height=18mm]{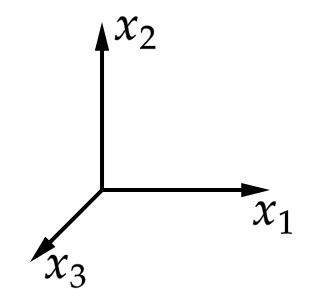}
         & 
         \includegraphics[width=20mm, height=20mm]{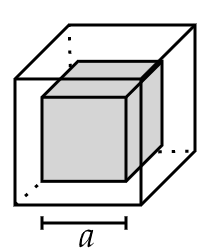}
         &
         \hspace{0pt}
         \includegraphics[width=20mm, height=20mm]{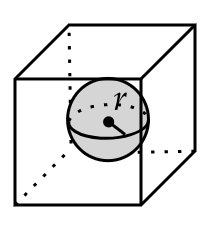}
         &
         \hspace{0pt}
         \includegraphics[width=20mm, height=20mm]{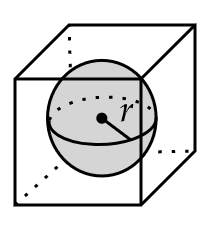}
         \\
         \hline
         &&& \\[\dimexpr-\normalbaselineskip+2pt]
         Geometry info & $a = 0.6764$ & $ r = 0.4196 $ & $r = 0.4673$ \\
         \hline&&& \\[\dimexpr-\normalbaselineskip+2pt]
        $|Y^f|$ & 0.6906 & 0.6906 & 0.5726 \\
        \hline
        $|\Gamma|$ & 2.7451 & 2.2125 & 2.7451 \\
        \hline
        &&& \\[\dimexpr-\normalbaselineskip+2pt]
        $\kappa^h / \kappa^f$ & $0.586 I$ & $0.596 I$ & $0.459 I$ \\
        \hline &&& \\[\dimexpr-\normalbaselineskip+2pt]
        $K$ & $0.008 I$ & $0.01 I$ & $0.005 I$
    \end{tabular}}
   \caption{Considered pore structures for the disconnected case. If only one side length is specified, the geometry is a cube. Here, $K$ denotes the permeability used in \cref{ch:simu_with_flow}.}
    \label{tab:data_connected}
\end{table}
\begin{table}[H]
    \centering
    \footnotesize{
    \begin{tabular}{c|c|c|c}
        Name tag & (C1) & (C2) & (C3) \\
        \hline
         \includegraphics[width=18mm, height=18mm]{Images/Axis.png}
         & 
         \includegraphics[width=24mm, height=24mm]{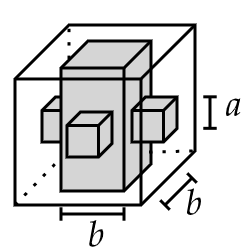}
         &
         \hspace{5pt}
         \includegraphics[width=24mm, height=24mm]{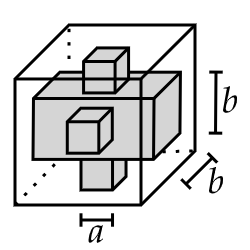}
         &
         \hspace{5pt}
         \includegraphics[width=24mm, height=24mm]{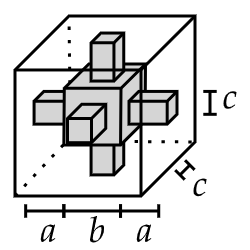}
         \\
         \hline
         &&& \\[\dimexpr-\normalbaselineskip+2pt]
         Geometry & $a = 0.2525,$ & $a = 0.2525,$& $a = \tfrac{1-b}{2}, c=0.3,$ \\
         info & $b=0.495$ & $b=0.495$& $b=0.5813$ \\
         \hline&&& \\[\dimexpr-\normalbaselineskip+2pt]
         $|Y^f|$ & 0.6906 & 0.6906 & 0.6906 \\
        \hline
        $|\Gamma|$, $|\Sigma|$ & 2.7451, 0.245 & 2.7451, 0.0638& 2.9948, 0.09 \\
        \hline
        &&& \\[\dimexpr-\normalbaselineskip+2pt]
        $\kappa^f_h / \kappa^f$  &
            \hspace{-4pt}$\begin{pmatrix}
                    0.501 & 0.0 & 0.0 \\
                    0.0 & 0.624 & 0.0 \\
                    0.0 & 0.0 & 0.501 
                 \end{pmatrix}$ \hspace{-8pt}
        &   \hspace{-8pt} $\begin{pmatrix}
                    0.624 & 0.0 & 0.0 \\
                    0.0 & 0.501 & 0.0 \\
                    0.0 & 0.0 & 0.501 
                    \end{pmatrix}$ \hspace{-8pt}
        &  $0.541 I$
        \\
        \hline
        &&& \\[\dimexpr-\normalbaselineskip+2pt]
        $\kappa^s_h / \kappa^s$  &
             \hspace{-10pt} $\begin{pmatrix}
                    0.093 & 0.0 & 0.0 \\
                    0.0 & 0.256 & 0.0 \\
                    0.0 & 0.0 & 0.093 
                 \end{pmatrix}$\hspace{-8pt}
        &   \hspace{-11pt}  $\begin{pmatrix}
                    0.256 & 0.0 & 0.0 \\
                    0.0 & 0.093 & 0.0 \\
                    0.0 & 0.0 & 0.093 
                    \end{pmatrix}$\hspace{-8pt}
        & $0.139 I$  
        \\
        \hline
        &&& \\[\dimexpr-\normalbaselineskip+2pt]
        $K$  &
            \hspace{-10pt} $\begin{pmatrix}
                    0.008 & 0.0 & 0.0 \\
                    0.0 & 0.013 & 0.0 \\
                    0.0 & 0.0 & 0.008 
                 \end{pmatrix}$\hspace{-8pt}
        &   \hspace{-11pt}  $\begin{pmatrix}
                    0.013 & 0.0 & 0.0 \\
                    0.0 & 0.008 & 0.0 \\
                    0.0 & 0.0 & 0.008 
                    \end{pmatrix}$\hspace{-8pt}
        &   $0.008 I$
    \end{tabular}}
    \caption{Considered pore structures for the connected case.}
    \label{tab:data_non_connected}
\end{table}

\subsection{Simulations without convection}\label{ch:simu_without_flow}
\paragraph{Stationary temperature profile.}\label{ch:stationary_profile}
We \menew{begin} by considering a stationary case to validate our model \tfnew{and implementation}, \menew{particularly with respect} to the varying diffusion values \menew{within} different subdomains.
In this section, we set all \tfnew{source} terms to zero and employ the Dirichlet conditions: $\thetaff = 0, \text{if } x_2 = 1 \text{ and } \theta^f = 1, \text{if } x_2 = 0$.
\menew{For} the connected model, we also enforce $-\kappa^s_h \nabla\theta^s \cdot \nu = \alpha |\Sigma|(1-\theta^s), \text{if } x_2 = 0$.
With these conditions, we obtain a one dimensional profile along the $x_2$ axis.
Therefore, we solve the problem given by \tfnew{\cref{eq:locall_problem}} only at the points $\{\mathbf{x}_i\}_{i=1}^{11}$, \menew{where} $\mathbf{x}_{i,1}=\mathbf{x}_{i,3}=1$ and $\mathbf{x}_{i,2}= \tfrac{i-1}{20}$, \menew{using linear interpolation between these points}.

The computed temperature is shown in \cref{fig:temp_sationary}.
In the disconnected case, \menew{$\Theta^s$ is also a function of \(y\). As a consequence, there is no unique temperature of the solid that can be illustrated along the \(x_2\) axis.
Hence}, we display the averaged temperature at point $x$
\begin{equation}\label{eq:average_temp_solid}
    \overline{\Theta}^s(x) = \frac{1}{|Y^s|} \int_{Y^s} \Theta^s(x, y) \di{y}.
\end{equation}
The expected \tf{slope change} of the fluid temperature \me{profile} can be seen in both models. 
One noticeable difference is the solid temperature profile:
In the disconnected case, fluid and solid temperature are \menew{identical, but} they differ in the connected \menew{model}.
\menew{This discrepancy arises due to heat transfer within the solid and heat exchange at the external boundaries ($\partial\Omega^{ff}\setminus\Sigma$ and $\partial\Omega^p\setminus\Sigma$) as well as on $\Sigma$ in the connected model.}

To \tfnew{illustrate} our homogenization results, simulations were also carried out for the resolved pore model (\ref{problem1}) using discontinuous linear finite elements.
\tfnew{Within $\Omega^{ff}$, the temperature disparity between the homogenized and resolved models, for both cases, is negligible.
In the solid domain $\Omega^p$, we observe that the effective fluid and solid temperatures pass through the temperature jumps of the resolved model.
This is an expected behavior of the homogenized solutions.
Minor variations are noted around the interface \(\Sigma\), which we attribute to the chosen position of \(\Sigma\), reminiscent of scenarios involving fluid flow over porous domains \cite{ER21}.}

Overall, we conclude that the simulations of this problem confirms our homogenization model in terms of effective heat conduction and exchange. 
Smaller $\varepsilon$ values could not be resolved on our \tfnew{workstation due to the intensive computational requirements, especially for mesh refinement studies concerning accuracy.
This limitation is why we restricted our comparison to this specific scenario and did not extend it to subsequent simulations.}   


\begin{figure}[ht]
    \centering
    \includegraphics[width=0.9\linewidth]{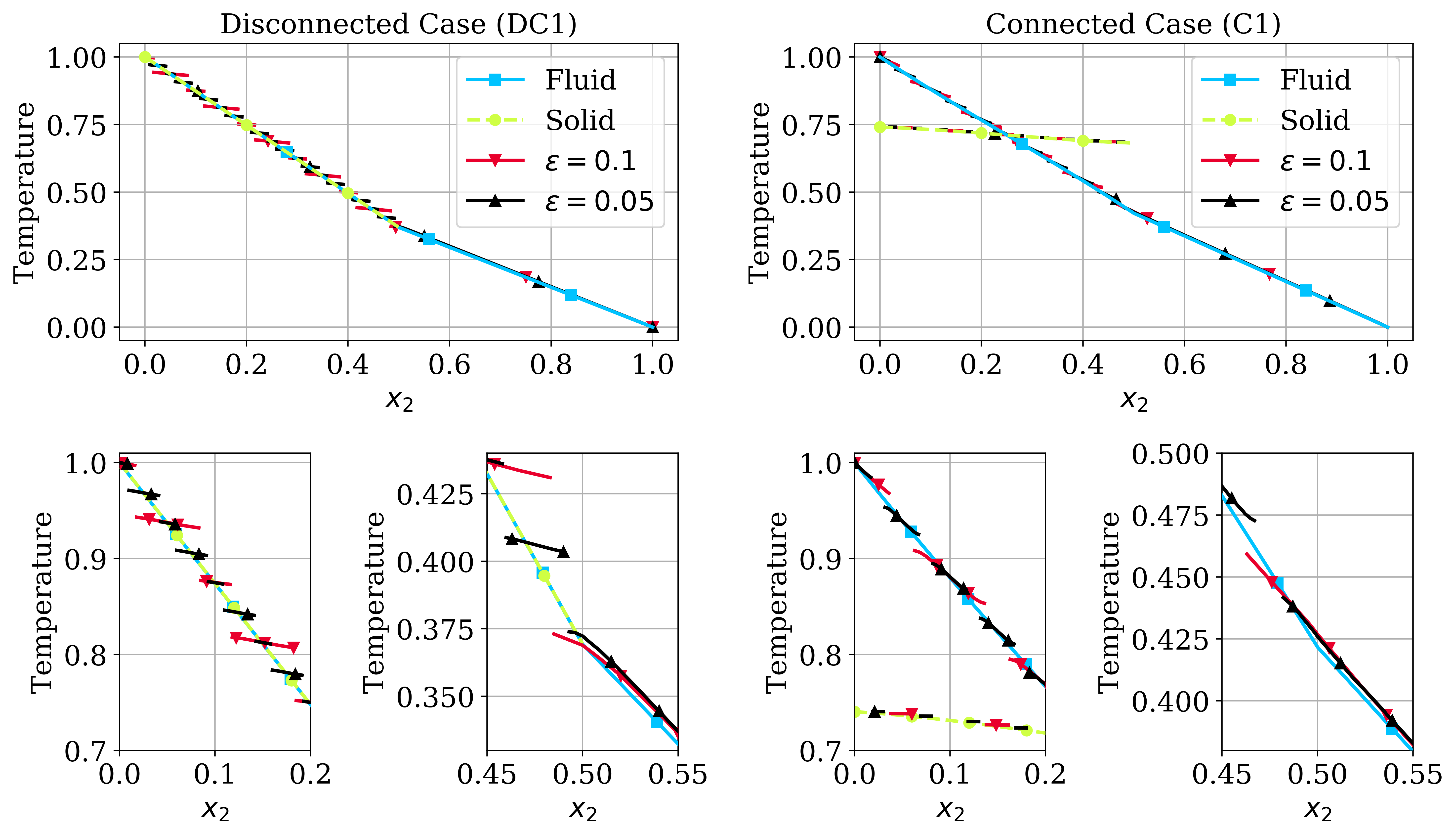}
    \vspace*{-12pt}
    \caption{\tf{Fluid and solid temperature fields along the $x_2$ axis for both homogenized models. 
    The pore structures mentioned in \cref{tab:data_connected,tab:data_non_connected} are used. 
    For comparison, the temperature profile of model (\ref{problem1}) with resolved pores is shown. 
    The solutions for (DC1) and (C1) are plotted along different lines, which cross both the fluid and the solid domain in each case. Other line positions yield similar comparisons results.}}
    \label{fig:temp_sationary}
\end{figure}
\paragraph{Influence of the pore structure.}
\tfnew{Motivated by our goal to simulate the temperature dynamics of} a grinding process, wherein heat generation occurs due to friction between the grinding wheel and material, we consider a \tfnew{heat source} at the interface:
\begin{equation*}
    f_\Sigma(t, x) = \begin{cases}
        1, & \text{if } t \leq 10.0 \text{ and } x \in \Sigma \\
        0, & \text{otherwise}
    \end{cases}.
\end{equation*}
\menew{This step function in time also helps} demonstrate the energy balancing process.
In the disconnected case, the production $f_\Sigma$ can be \menew{directly} added to the interface \cref{e:hom1}.
In the case of a connected structure, we \menew{partition} the \tfnew{source term} between fluid and solid \menew{using} the following conditions:
\begin{alignat*}{2}
-\kappa^f\nabla\theta^{ff}\cdot n_\Sigma&= - \kappa_{h}^f\nabla\theta^{f}\cdot n_\Sigma + (1-|\Sigma^s|)f_\Sigma + \alpha|\Sigma^s| (\theta^s-\theta^f)&\quad&\text{on}\ \ S\times\Sigma,\\
-\kappa_h^s\nabla\theta^s\cdot n_\Sigma&=\alpha|\Sigma^s|(\theta^s-\theta^f) - |\Sigma^s|f_\Sigma&\quad&\text{on}\ \ S\times\Sigma.
\end{alignat*}
In the weak formulation, the source term $f_\Sigma$ appears as an integral over $\Sigma$.
Homogeneous Neumann conditions are applied at all external boundaries. 

These condition \menew{result in} a one dimensional temperature profile.
For this reason, we use the same points $\{\mathbf{x}_i\}_{i=1}^{11}$ as in the previous section.
For comparison, we \menew{assess}, at different \menew{points in time} $t$, the temperature along the $x_2$ axis and the \tf{heat} energy \tf{(enthalpy)} inside the subdomains, which correspond to
\begin{equation}\label{eq:heat_energy} 
       E^{ff}(t) = \int_{\Omega^{ff}} \tf{\rho^f c^f} \theta^{ff}(t, x) \di{x}, \quad
    E^{r}(t) = |Y^r| \int_{\Omega^{p}}\rho^r c^r \theta^{r}(t, x) \di{x} \quad (r=f,s).
\end{equation}
The energy and temperature are presented in \cref{fig:temp_pore_structures} and \ref{fig:energy_pore_structures}.
For pore geometries (DC1) and (DC2), where the porosity $|Y^f|$ is \menew{identical}, the temperature and energy profiles \menew{closely align}.
For (DC2), the solid is \menew{slightly} cooler \menew{as a consequence of the smaller interface} $\Gamma$.
The structure (DC3) yields more \menew{pronounced} deviations, \menew{mainly} because of the difference in $|Y^f|$.

The models for connected and disconnected matrices also produce noticeable differences, even \menew{with identical} geometry parameters.
\menew{This is evident} when comparing the results for (DC1) and (C1).
\menew{The differences arise} due to the interface production on the solid and \menew{the potential} for heat diffusion \menew{in the solid phase} in the connected case.
Varying the sizes of $|\Gamma|$ and $|\Sigma|$ \tfnew{produces} expected results, since smaller values correspond to slower heat transfer into the solid.

\tfnew{To further validate our numerical approach, we verify energy conservation}.
\menew{Given that our} system is isolated and $f_\Sigma$ is known, the sum of the terms in \cref{eq:heat_energy} should match the integral over $f_\Sigma$.
\menew{Pertinently}, all simulation cases \menew{exhibit a discrepancy} less than 0.25\% from the expected value at every time step.
\begin{figure}[H]
    \centering
    \includegraphics[width=\linewidth]{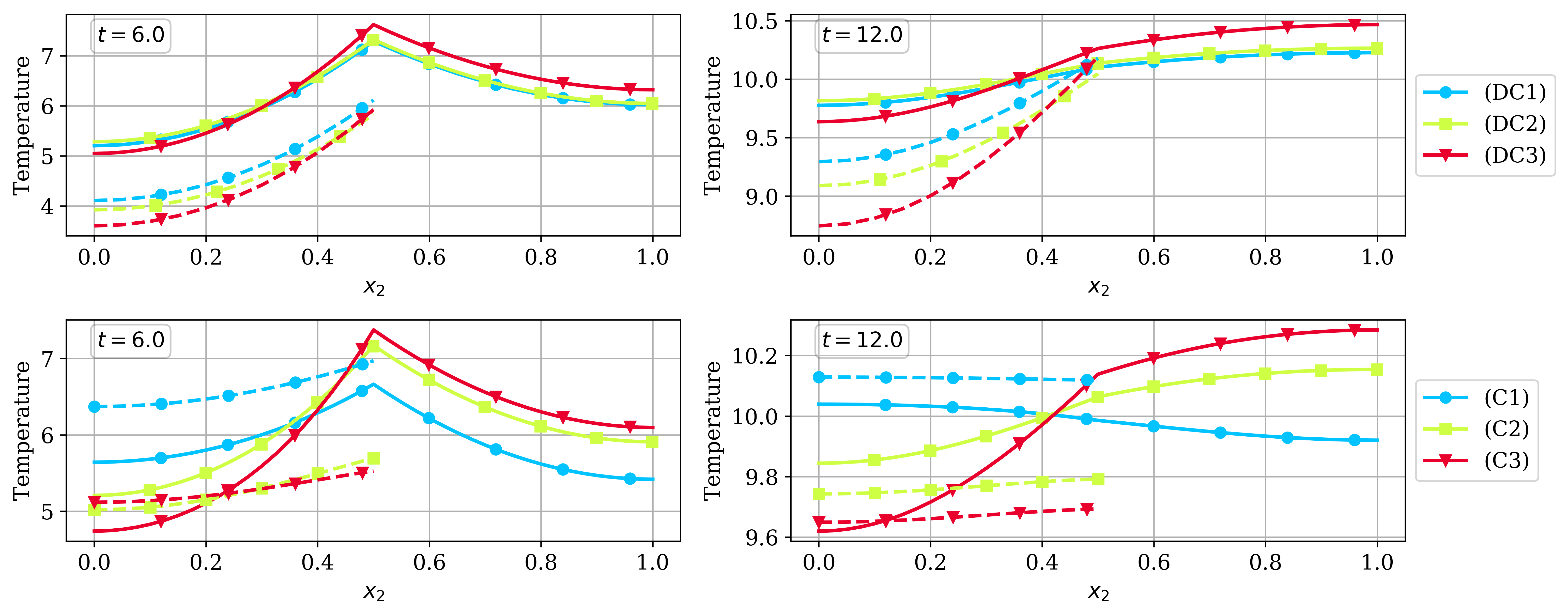}
    \vspace*{-25pt}
    \caption{\tfnew{Temperature profiles, along the $x_2$ axis, for different pore structures and at two distinct snap shots in time. The dashed lines depict the temperature of the solid. \menew{The top row displays} the disconnected case, \menew{while the bottom illustrates} the connected one.}}    
    \label{fig:temp_pore_structures}
\end{figure}
\vspace{-10pt}
\begin{figure}[H]
    \centering
    \includegraphics[width=\linewidth]{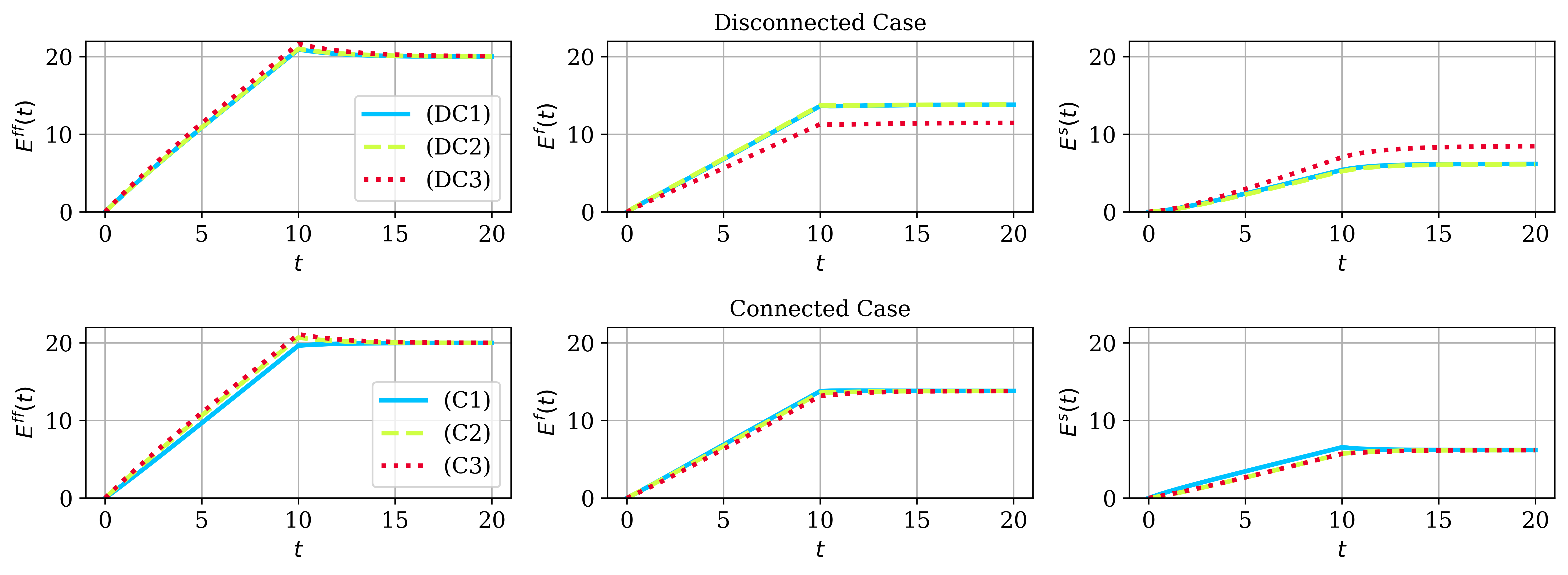}
    \vspace*{-25pt}
    \caption{\tfnew{Heat energies (\ref{eq:heat_energy}) \menew{within} different subdomains. All introduced pore structures in \cref{tab:data_connected,tab:data_non_connected} are represented.}}
    \label{fig:energy_pore_structures}
\end{figure}

\paragraph{Transition from connected to disconnected case.}\label{ssec:transition}
One question \me{that arises is} whether the results for the connected matrix \tfnew{do} approach those of the disconnected case when the connected pore structure transitions to a disconnected one.
To examine this aspect, we utilize the pore geometries (DC1) and (C3) reducing the side length $c$ while increasing $b$ and keeping $|Y^f|$ fixed.
The parameters for different values of $c$, are listed in \cref{tab:limit_c}.

\menew{For} comparison, we compute the temperature difference between both models \menew{using} the $L^2$-norm at each time step.
For the solid temperature, we again \menew{employ} the average (\ref{eq:average_temp_solid}).
The results are shown in \cref{fig:transition}, where the expected trend of the two models yielding comparable solutions for $c\to 0$ is clearly visible.
\menew{Interestingly, even setting \(c=0\) in model \ref{sec:con} yields plausible results}. 
The \menew{minor} relative difference of approximately $10^{-3}$ \menew{mainly stems} from numerical errors \menew{related to} energy conservation.

\begin{table}[H]
    \centering
    \begin{tabular}{c|c|c|c|c|c}
         $c$& $b$ & $|\Gamma|$ & $|\Sigma|$ & $\kappa^f_h/\kappa^f$ & $\kappa^s_h/\kappa^s$\\
         \hline
         0.2  & 0.6437 & 3.1012 & 0.04 & \tf{$0.564 I$} & \tf{$0.078 I$} \\ \hline
         0.1  & 0.6691 & 3.0232 & 0.01 & \tf{$0.581 I$}  & \tf{$0.025 I$} \\ \hline
         0.05 & 0.6746 & 2.9108 & 0.0025 & \tf{$0.584 I$} & \tf{$0.006 I$} \\ \hline
         0.0  & 0.6764 & 2.7451 & 0.0 & \tf{$0.586 I$}  & \tf{-}
    \end{tabular}
    \caption{Parameters of (C3) for different side lengths $c$.}
    \label{tab:limit_c}
\end{table}
\begin{figure}[h]
    \centering
    \includegraphics[width=\linewidth]{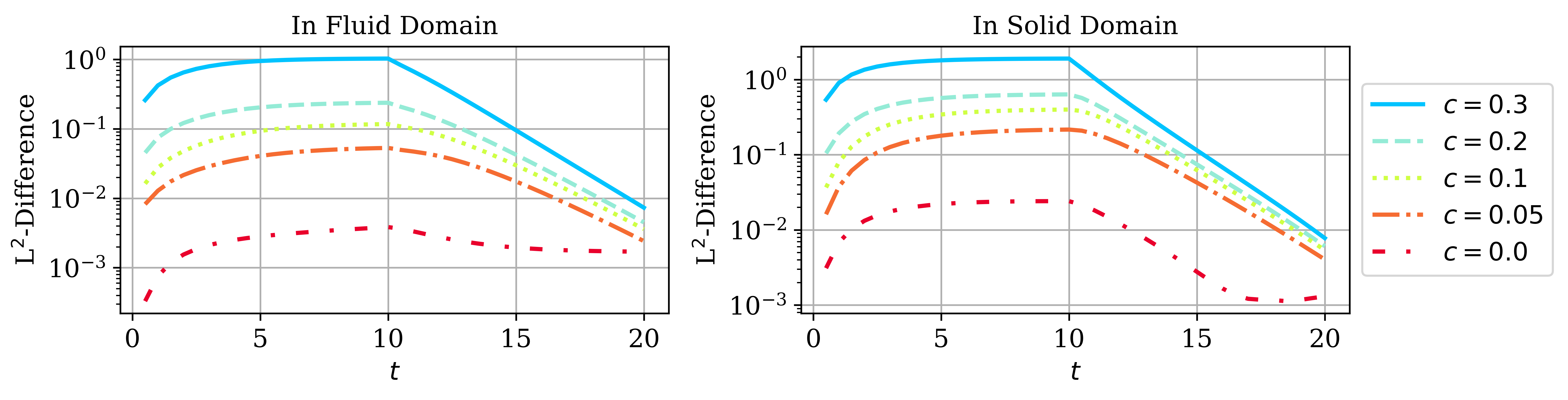}
    \vspace*{-16pt}
    \caption{\tfnew{The $L^2$-difference in temperature between the solution of the connected and disconnected cases.
    \menew{We contrast} pore geometry (DC1) with geometry (C3) for different values of $c$.
    \menew{The left side showcases} the fluid domain $\Omega^{ff}$, \menew{while} the right \menew{presents} differences inside $\Omega^{p}$.}}
    \label{fig:transition}
\end{figure}
\subsection{Simulations with convection}\label{ch:simu_with_flow}
Finally, we \menew{integrate} fluid flow in the simulation studies.
\menew{Drawing inspiration from engineering applications where the fluid functions as a coolant}, we \me{specified} an inflow temperature and velocity.
Additionally, we \menew{disregard} the influence of the temperature on the flow (e.g., buoyancy) and \menew{assume} a stationary flow profile established \menew{prior} to heating.
An essential parameter for the Darcy equation is the permeability tensor $K$.
For a given pore structure, the permeability can be computed via the solutions of problems inside the cell $Y^f$ \cite{H97}.
\tfnew{The permeability matrices for our chosen geometries} are listed in \cref{tab:data_connected,tab:data_non_connected}.

\menew{To streamline both the model and our discussion}, we consider a two-dimensional flow and temperature profile \menew{invariant} in the $x_3$ direction.
For the inflow, we \menew{prescribe} Dirichlet conditions $u^{ff} = (16(x_2-0.5)(1-x_2), 0, 0)$ and $\thetaff = 0$, if $x_1=0$  and $x_2 \geq 0.5$.
On the opposite boundary at $x_1=2$  and $x_2 \geq 0.5$, we apply a free outflow condition, $(\tfrac{\mu}{2}(\nabla u^{ff} + (\nabla u^{ff})^T) - p^{ff}I) \cdot \nu = 0$, and only allow convective heat transport, $\kappa^f \nabla \thetaff \cdot \nu = 0$.
At all other boundaries a no-slip condition for the flow and homogeneous Neumann condition for the temperature are used.
The viscosity is set to $\mu = 1$.
\menew{Finally, an oscillating heat source, defined by \(\Tilde{f}_\Sigma(t, x) = f_\Sigma(2(t \mod{10}), x)\), is applied over the extended time interval \(S=[0, 50]\).}

\begin{remark}
    \menew{In this simulation, only flow parallel to the interface is considered.
    Although our heat exchange model applies to general cases with arbitrary flow directions, non-parallel flows might necessitate modifications to the Beavers–Joseph conditions, incorporating additional terms.
    For further details, see \cite{ER21} and the references therein.}
\end{remark}

Again, both connected and disconnected models are simulated, \menew{specifically} the geometries (C1) and (DC1).
Since the temperature profile will also vary in the $x_1$ direction, we solve the cell problems (\ref{eq:locall_problem}) at the points $\{\mathbf{x}_{ij}\}_{i,j=1}^{11}$, with $\mathbf{x}_{ij,1}= \tfrac{j-1}{5}, \mathbf{x}_{ij,2}= \tfrac{i-1}{20}$ and $\mathbf{x}_{ij,3}=1$.

The resulting velocity and temperature profiles, \menew{captured} at $t=45$, for both model types are presented in \cref{fig:flow_and_temperature}.
\menew{While fluid velocity variations between the two models are minimal, especially \menew{within} $\Omega^{ff}$, the temperature profiles differ markedly.}
In the \menew{connected case}, the solid is hotter than in the disconnected case. 
\menew{Here, the connected geometry facilitates heat diffusion within the solid, yielding a pronounced counter-effect against heat convection within the fluid of \(\Omega^p\).}

We also study the \menew{effects} of the heat exchange \tfnew{parameter} $\alpha$.
\menew{As anticipated,} a larger $\alpha$ leads to similar temperature profiles in fluid and solid, \menew{while} smaller values may allow for heat to accumulate in the solid domain. 
The corresponding energy curves for different $\alpha$ values are displayed in \cref{fig:energy_for_alpha}.
\menew{Particularly in the connected system, the energies precisely mirror the oscillations of the heat source \(\Tilde{f}_\Sigma\).}
For small values of $\alpha$ the solid can heat up considerably due to the production on $\Sigma$ \menew{thereby diminishing the cooling effect of the coolant.}
For the disconnected system, the solid is \menew{generally} a lot cooler than \tfnew{in} the connected model.
Interesting are the cases $\alpha=0.01$ or $0.1$, where the oscillations of heat inside the solid are lagging behind the oscillation of $\Tilde{f}_\Sigma$ and the energy inside the fluid.
Here, we can see the memory effect present in the one-temperature model given in \cref{hom:b2}.
For large $\alpha$ both models exhibit similar results, since \menew{a increased} heat exchange decreases the ability to store heat inside the solid and, in the connected model, also \menew{dampens} the impact of heat diffusion in the solid.

\begin{figure}[H]
     \centering
     \begin{subfigure}[b]{0.32\textwidth}
         \centering
         \includegraphics[width=\textwidth]{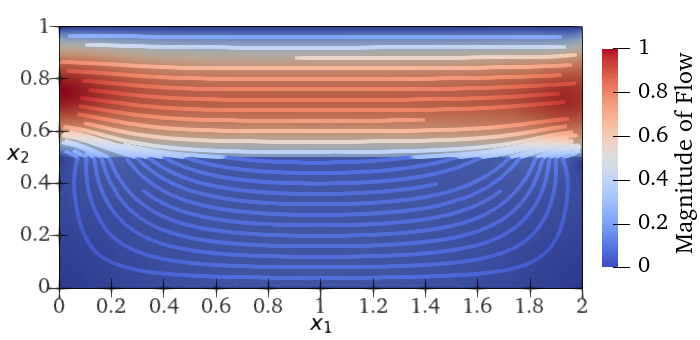}
     \end{subfigure}
     \hfill
     \begin{subfigure}[b]{0.32\textwidth}
         \centering
         \includegraphics[width=\textwidth]{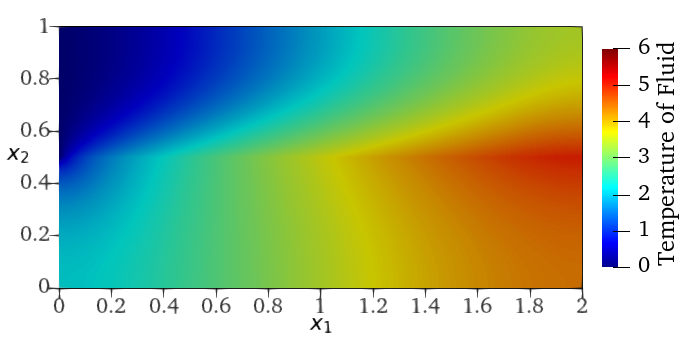}
     \end{subfigure}
     \hfill
     \begin{subfigure}[b]{0.32\textwidth}
         \centering
         \includegraphics[width=\textwidth]{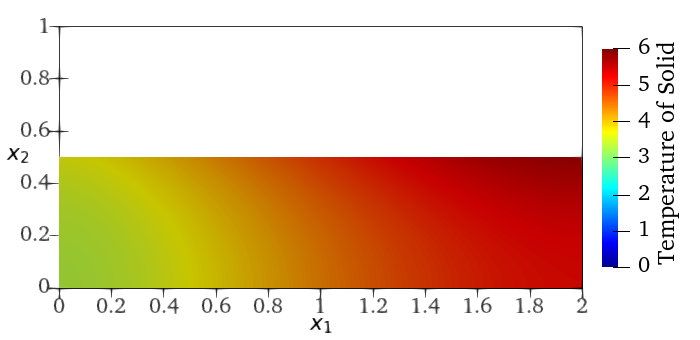}
     \end{subfigure} \\
     \begin{subfigure}[b]{0.32\textwidth}
         \centering
         \includegraphics[width=\textwidth]{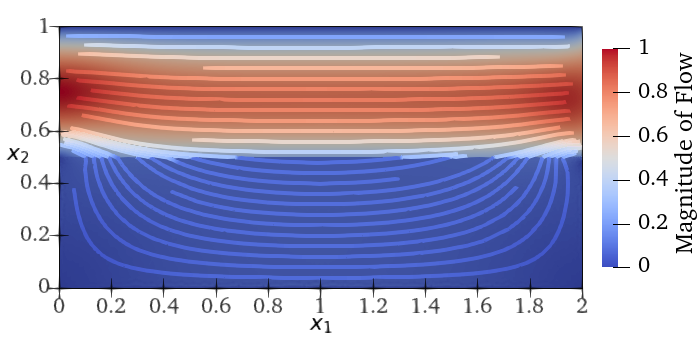}
     \end{subfigure}
     \hfill
     \begin{subfigure}[b]{0.32\textwidth}
         \centering
         \includegraphics[width=\textwidth]{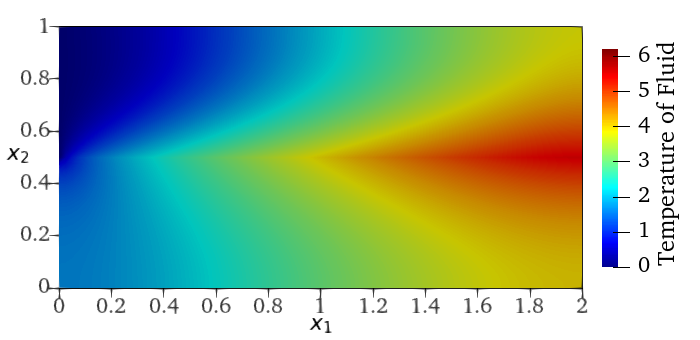}
     \end{subfigure}
     \hfill
     \begin{subfigure}[b]{0.32\textwidth}
         \centering
         \includegraphics[width=\textwidth]{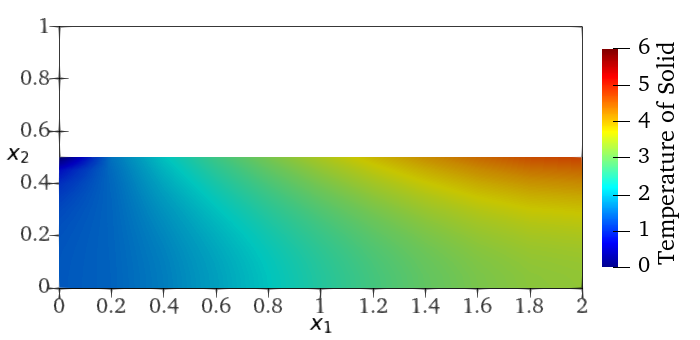}
     \end{subfigure}
   \caption{Cross-section of the simulation results. Left: magnitude and streamlines of the flow. Middle and right: the (average) temperature profile at $t=45$ and $\alpha=0.1$ for the fluid and solid domain respectively. The top row corresponds to the connected system the bottom one to the disconnected system.}
    \label{fig:flow_and_temperature}
\end{figure}

\begin{figure}[!h]
    \centering
    \includegraphics[width=\linewidth]{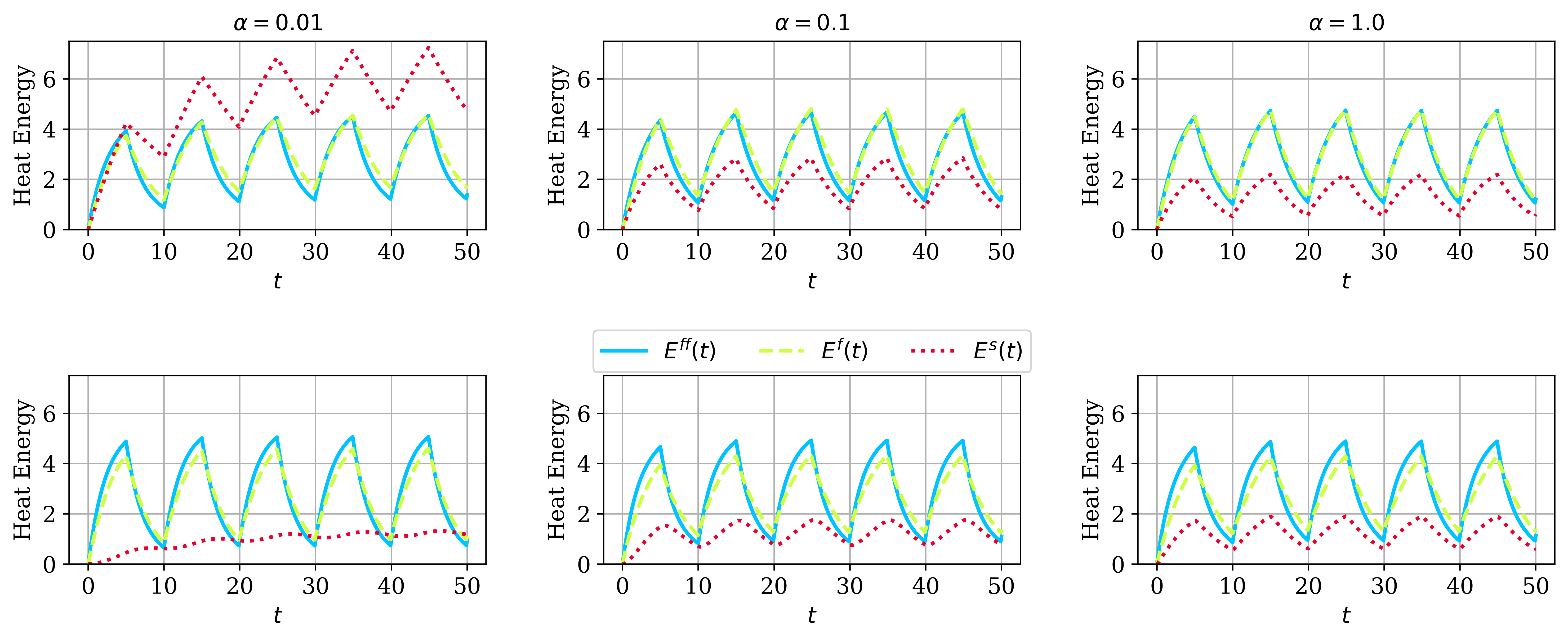}
    \caption{\tfnew{Temporal development of the heat energies (\ref{eq:heat_energy}). For the case of an underlying convection and oscillating heat source. Each column shows the results for a specific heat exchange parameter $\alpha$. The top row corresponds to the connected system the one at the bottom to the disconnected system.}}
    \label{fig:energy_for_alpha}
\end{figure}

\section*{Acknowledgements}
This research was funded by the Deutsche Forschungsgemeinschaft (DFG, German Research Foundation) -- project nr. 439916647.

The research activity of ME is funded by the European Union’s Horizon 2022 research and innovation program under the Marie Skłodowska-Curie fellowship project {\em{MATT}} (project nr.~101061956).
TF acknowledges funding by the Deutsche Forschungsgemeinschaft (DFG) -- project nr. \\ 281474342/GRK2224/2.


\bibliographystyle{siam}
\bibliography{references}

\appendix
\section{\tfnew{Proof of \cref{rem:algo_conv}}}\label{sec:algo_conv}
\begin{proof}
        The values of $\theta^{ff}_{n}, \theta^{f}_{n}, \theta^{s}_{n}$ at the previous time step are known. For the fixed point algorithm we define the iteration step $k$, solutions at step $k$ with  $\theta^{ff}_{k, n+1}, \theta^{f}_{k, n+1}, \theta^{s}_{k, n+1}$ and the difference 
        \begin{equation*}
            e_k^{r} = \theta^{r}_{k+1, n+1} - \theta^{r}_{k, n+1}, \text{ for } r=ff,f,s.
        \end{equation*}
        At iteration $k$, the weak formulation for the cell problems is
        \begin{equation*}
            \begin{split}
                 \int_{\Omega_p  \times Y^s} \tfrac{\rho^s c^s}{\Delta t} &\theta^{s}_{k, n+1} \varphi \di{(x, y)} 
                 + \int_{\Omega_p  \times Y^s} \kappa^s \nabla_y \theta^{s}_{k, n+1} \nabla_y \varphi \di{(x, y)} 
                 + \int_{\Omega_p  \times \Gamma} \alpha \theta^{s}_{k, n+1} \varphi \di{(x, \sigma)} \\
                &= \int_{\Omega_p  \times Y^s} f^{s}_{n+1} \varphi \di{(x, y)} 
                + \int_{\Omega_p  \times \Gamma} \alpha \theta^{f}_{k, n+1}\varphi \di{(x, \sigma)} 
                + \int_{\Omega_p  \times Y^s} \tfrac{\rho^s c^s}{\Delta t}\theta^{s}_{n} \varphi \di{(x, y)}.  
            \end{split}
        \end{equation*}
        Computing the difference between two following iterations, choosing the test function $\varphi = e_k^s$ and using Hölder's inequality on the remaining integral on the right side, leads to the estimate
        \begin{equation}\label{eq:appendix_estimate_s}
            \|e_k^s\|_{L^2(\Omega_p \times \Gamma)} \leq \sqrt{|\Gamma|} \|e_k^f\|_{L^2(\Omega_p)}.
        \end{equation}
        A similar computation (e.g. weak formulation, computing the difference of following iterations and testing with $e_k$) for the fluid temperature gives
        \begin{equation}\label{eq:appendix_fluid_problem}
            \begin{split}
                &\int_{\Omega^{ff}} \tfrac{\rho^f c^f}{\Delta t} e_k^{ff} e_k^{ff} \di{x} 
                + \int_{\Omega^{p}} |Y^f| \tfrac{\rho^f c^f}{\Delta t} e_k^{f} e_k^{f} \di{x} 
                + \int_{\Omega^{ff}} \kappa^f \nabla e_k^{ff} 
                \nabla e_k^{ff} \di{x} 
                + \int_{\Omega^{p}} \kappa^f_h \nabla e_k^{f} 
                \nabla e_k^{f} \di{x} \\
                &- \int_{\Omega^{ff}} \rho^f c^f v e_k^{ff} \nabla e_k^{ff} \di{x} 
                - \int_{\Omega^{p}} \rho^f c^f \Bar{v}_{D} e_k^{f} \nabla e_k^{f} \di{x} + \alpha |\Gamma| \int_{\Omega^p} e_k^{f} e_k^{f} \di{x} = \alpha \int_{\Omega_p \times \Gamma} e_{k-1}^s e_k^f \di{(x, \sigma)}.
            \end{split}
        \end{equation}
        Using \tfnew{the Hölder inequality} and the estimate (\ref{eq:appendix_estimate_s}) on the right hand side, we get
        \begin{equation*}
            \alpha \int_{\Omega_p \times \Gamma} |e_{k-1}^s e_k^f| \di{(x, \sigma)} \leq \alpha |\Gamma| \|e_{k-1}^f\|_{L^2(\Omega^p)} \|e_{k}^f\|_{L^2(\Omega^p)}.
        \end{equation*}
        For the left side of (\ref{eq:appendix_fluid_problem}), we define $\widetilde{e}_k = \chi_{\Omega^{ff}}e_k^{ff} + \chi_{\Omega^{p}}e_k^{p}$, $\widetilde{v} = \chi_{\Omega^{ff}}v + \chi_{\Omega^{p}}\Bar{v}_D$ and let $\lambda > 0$ be the \tfnew{coercivity} constant from \cref{rem:algo_conv}. Then we get the estimate
        \begin{equation*}
            \tfrac{\rho^f c^f |Y^f|}{\Delta t} \|\widetilde{e}_k \|_{L^2(\Omega)}^2
            + \lambda \|\nabla \widetilde{e}_k\|_{L^2(\Omega)}^2
            + \alpha |\Gamma| \|e_k^f \|_{L^2(\Omega^p)}^2
            - \int_{\Omega} c^f \rho^f \widetilde{v} \widetilde{e}_k \nabla \widetilde{e}_k \di{x} 
            \leq \alpha |\Gamma| \|e_{k-1}^f\|_{L^2(\Omega^p)} \|e_{k}^f\|_{L^2(\Omega^p)}.
        \end{equation*}
        \tflast{By Assumption (A4) we have $\|v_\e\|_{L^\infty(S\times \Omega^f)} \leq C_v$ independent of $\e$. By Assumption (A7), $v_\varepsilon \to v$ strongly in $L^2(S\times \Omega^{ff})$ and $\Tilde{v}_\e \rightharpoonup \Bar{v}_D$ weakly in $L^2(S \times \Omega^p)$. Since strong and weak convergence preserve pointwise estimates for almost all $x$, we have $\widetilde{v} \in {L^\infty}(S\times \Omega)$.}
        \tfnew{Using} on the convection term that $\widetilde{v} \in L^\infty(S \times \Omega)$ \tfnew{and applying} Hölder's and Young's inequality with $\delta < 2\lambda$ gives
        \begin{equation}\label{eq:appendix_final}
            \begin{split}
                \left(  \tfrac{\rho^f c^f |Y^f|}{\Delta t} - \tfrac{(\rho^f c^f)^2 \|\widetilde{v}\|_{L^\infty(S \times \Omega)}^2}{2\delta} \right) \|\widetilde{e}_k \|_{L^2(\Omega)}^2 
                &+ (\lambda - \tfrac{\delta}{2}) \|\nabla \widetilde{e}_k\|_{L^2(\Omega)}^2 + \alpha |\Gamma| \|e_k^f \|_{L^2(\Omega^p)}^2 \\
                &\leq \alpha |\Gamma| \|e_{k-1}^f\|_{L^2(\Omega^p)} \|e_{k}^f\|_{L^2(\Omega^p)}.
            \end{split}
        \end{equation}
        This immediately leads to the error estimate inside of $\Omega^p$ of the form
        \begin{equation*}
            \|e_{k}^f\|_{L^2(\Omega^p)} \leq \left( \frac{\alpha |\Gamma|}{\alpha |\Gamma| + \tfrac{\rho^f c^f |Y^f|}{\Delta t} - (\rho^f c^f)^2 \|\Tilde{v}\|_{L^\infty(S \times \Omega)}^2 \tfrac{1}{2\delta}} \right ) \|e_{k-1}^f\|_{L^2(\Omega^p)}.
        \end{equation*}
        With our condition on $\Delta t$, the bracket term is positive and smaller than one. Therefore, iterative application of the argument leads to convergence to a fixed point inside $\Omega^p$. Plugging \tfnew{the last inequality} back into \cref{eq:appendix_final} gives the desired error estimate and convergence in the whole domain $\Omega$.
\end{proof}
\begin{remark}\label{remark_consistency_error}
        To simplify the proof, we assumed that the cell problems (\ref{eq:locall_problem}) for $\theta^s$ are solved everywhere in $\Omega^p$.                
        For the numerical algorithm we solve only on specific points and interpolate the results with an linear operator $\Theta^s$.
        \tflast{One should note, that the use of the interpolation operator leads to a consistency error, which, in addition to the discretization errors of the cell problems and the macroscale domain, influences the accuracy of the numerical solution. In subsequent studies, the influence of the interpolation operator and the interaction with the discretization errors could be further investigated and possible error bounds be determined.}
\end{remark}
\end{document}